\documentclass{amsart}

\usepackage{enumerate}
\usepackage{amssymb, amsmath}
\usepackage{mathrsfs}
\usepackage[active]{srcltx}
\usepackage{verbatim}
\usepackage{color}

\usepackage[colorlinks,linkcolor={blue},citecolor={blue},urlcolor={red},]{hyperref}

\usepackage[fleqn,tbtags]{mathtools}
\mathtoolsset{showonlyrefs,showmanualtags}

\theoremstyle{plain}
\newtheorem{theorem}{Theorem}[section]
\newtheorem{corollary}[theorem]{Corollary}
\newtheorem{lemma}[theorem]{Lemma}
\newtheorem{proposition}[theorem]{Proposition}

\theoremstyle{remark}
\newtheorem{remark}[theorem]{Remark}
\newtheorem{definition}[theorem]{Definition}

\newtheorem{example}[theorem]{Example}
\newtheorem{assumption}[theorem]{Assumption}

\numberwithin{equation}{section}

\def\N{{\mathbb N}}

\def\R{{\mathbb R}}
\def\C{{\mathbb C}}

\newcommand{\E}{{\mathbb E}}
\renewcommand{\P}{{\mathbb P}}

\newcommand{\calH}{{\mathscr{H}}}

\renewcommand{\a}{\alpha}
\renewcommand{\b}{\beta}

\newcommand{\sgn}{{\rm sgn}}

\newcommand{\wh}{\widehat}
\newcommand{\supp}{{\rm supp}}
\newcommand{\dist}{{\rm dist}}

\renewcommand{\L}{\mathscr{L}}

 \newcommand{\bma}{\begin{bmatrix}}
 \newcommand{\ema}{\end{bmatrix}}

\renewcommand{\Re}{\hbox{\rm Re}}

\newcommand{\wt}{\widetilde}
\newcommand{\calL}{{\mathscr L}}
\newcommand{\n}{\Vert}
\newcommand{\one}{{{\bf 1}}}
\newcommand{\embed}{\hookrightarrow}
\newcommand{\s}{^{*}}
\newcommand{\lb}{\langle}
\newcommand{\rb}{\rangle}

\newcommand{\ov}{\overline}

\newcommand{\D}{\mathscr{D}}
\newcommand{\uL}{\underline{\L}}
\newcommand{\A}{{\mathscr A}}
\newcommand{\uC}{\underline{C}}
\newcommand{\uS}{\underline{S}}
\newcommand{\uD}{\nabla_H}

\newcommand{\uR}{\underline{R}}
\newcommand{\Sin}{S}
\newcommand{\uSin}{\underline S}
\newcommand{\Cos}{C}

\newcommand{\limn}{\lim_{n\to\infty}}

\newcommand{\Dom}{{\mathsf D}}
\newcommand{\Ran}{\mathsf{R}}
\newcommand{\Null}{\mathsf{N}}

\newcommand{\dt}{{\rm d}t}
\newcommand{\ds}{{\rm d}s}
\newcommand{\dmu}{{\rm d}\mu}

\newcommand{\etaA}{\eta}
\newcommand{\etaB}{\eta}

\usepackage{newsymbol}

\let\mathcal \undefined
\def\mathcal{\mathscr}

\let\emptyset \undefined
\let\ge       \undefined
\let\le       \undefined
\newsymbol\le          1336  \let\leq\le
\newsymbol\ge          133E  \let\geq\ge
\newsymbol\emptyset    203F
\newsymbol\notle       230A
\newsymbol\notge       230B

\allowdisplaybreaks

\begin{document}

\title[The Hodge-Dirac operator associated with the O-U operator]
{Finite speed of propagation and off-diagonal bounds for Ornstein-Uhlenbeck operators in infinite dimensions}

\author{Jan van Neerven}

\address{Delft Institute of Applied Mathematics\\
Delft University of Technology\\P.O. Box 5031\\2600 GA Delft\\The Netherlands.}
\email{J.M.A.M.vanNeerven@tudelft.nl}

\author{Pierre Portal}
\address{Australian National University, Mathematical Sciences Institute, \, John Dedman Building, 
Acton ACT 0200, Australia\\
and Universit\'e Lille 1, Laboratoire Paul Painlev\'e, F-59655 {\sc Villeneuve d'Ascq}, France.}
\email{Pierre.Portal@anu.edu.au}         

\begin{abstract}
We study the Hodge-Dirac operators $\D$ associated with a class of non-symmetric Ornstein-Uhlenbeck
operators $\calL$ in infinite dimensions. For $p\in (1,\infty)$ we prove 
that $i\D$ generates a $C_0$-group in $L^p$ with respect to the
invariant measure if and only if $p=2$ and 
$\L$ is self-adjoint. An explicit representation of this $C_0$-group in $L^2$ is given 
and we prove that it has finite speed of propagation.
Furthermore we prove $L^2$ off-diagonal estimates for
various operators associated with $\L$, both in the self-adjoint and the non-self-adjoint case.
\end{abstract} 

\keywords{Ornstein-Uhlenbeck operator, Hodge-Dirac operator, $C_0$-group, finite speed of propagation, 
heat kernel bounds, Davies-Gaffney estimates}

\subjclass[2000]{47A60, 47F05, 60H15, 42B37, 35L05}

\date{\today}

\thanks{This work was supported by the Australian Research Council through the Discovery Project 
DP120103692, as well as through van Neerven's NWO-VICI subsidy 639.033.604 and Portal's Future Fellowship FT130100607.}

\maketitle

\section{Introduction}
In this paper we establish analogues of several well-known $L^p$-results for the wave group
$(e^{it\sqrt{-\Delta}})_{t\ge 0}$, the Schr\"odinger group $(e^{it\Delta})_{t\ge 0}$,  
and the heat semigroup $(e^{t\Delta})_{t \geq 0}$
by replacing the Laplace operator $\Delta$ by a 
(possibly infinite-dimensional and non-symmetric) Ornstein-Uhlenbeck operator. 
Our principal tool is the first-order approach introduced by Axelsson, Keith, and M$^{\rm c}$Intosh 
\cite{AKM} and developed in many recent papers 
\cite{AA11, AAH08, AAM10-2, AAM10, AMR08, AA06, FMP14, HMP08, HMP11, McIMor}, 
which looks at these objects through the 
functional calculus of {\em Hodge-Dirac operators} such as 
\begin{align}\label{eq:diracD} 
D := \bma 0 & -\,{\rm div} \\ \nabla & 0\ema 
\end{align}
acting on the direct sum $L^2(\R^d)\oplus L^2(\R^d;\C^d)$.
This approach has already been used in the Ornstein-Uhlenbeck context in 
\cite{MN,MaaNee11} to obtain necessary and sufficient 
conditions for the $L^p$-boundedness of Riesz transforms. 
The relevant Hodge-Dirac operator is given by 
\begin{equation*}
\D := \bma 0 & \nabla_H\s B \\ \nabla_H & 0\ema,
\end{equation*}
acting on $L^{2}(E,\mu)\oplus L^{2}(E,\mu;H)$, where
$E$ is a Banach space, $\mu$ is  an invariant measure on $E$,  
$H$ is a Hilbert  subspace of $E$,  
 $\nabla_{H}$ is the gradient in the direction of $H$, and 
$B$ is a bounded linear operator acting on $H$ (see Section \ref{sec:OU} for precise definitions).
The corresponding Ornstein-Uhlenbeck operator is then given by 
\begin{align*}
\L = -\frac12\nabla_H\s B \nabla_H.
\end{align*}
\medskip
The first result we prove is a version for Ornstein-Uhlenbeck operators of
the following theorem on $L^p$-extendability of the wave group.
It can be viewed as an analogue of the classical result of H\"ormander \cite{Hor} 
(see also \cite[Theorem 3.9.4]{ABHN}) stating 
that the Schr\"odinger group $(e^{it\Delta})_{t\in\R}$
extends to $L^p(\R^d)$ if and only if $p=2$.

\begin{theorem}\label{thm:dirac} Let $1< p<\infty$ and $d\ge 1$. The following assertions are equivalent:
\begin{enumerate}[\rm(i)]
 \item the operator $i\sqrt{-\Delta}$ generates a $C_0$-group on $L^p(\R^d)$;
 \item $p=2$ or $d=1$.
\end{enumerate}
\end{theorem}

This equivalence 
is due to Littman \cite{Litt}; a proof by Fourier multiplier methods
can be found in \cite[Theorem 8.3.13]{ABHN}. 

Theorem \ref{thm:dirac} shows that, even in the setting of $\R^{n}$
and the Euclidean Laplacian, simple oscillatory Fourier 
multipliers can fail to be bounded in $L^p$ for $p\neq 2$. The study of such 
operators that are beyond the reach of classical results on Fourier multipliers 
such as the Mihlin-H\"ormander theorem, is an important objective of Fourier 
integral operator theory. One of the first results in this direction is the 
following theorem of Miyachi \cite[Corollary 1]{Miy} and Peral \cite{Per}, that shows that a suitably regularised 
version of the wave group is $L^p$-bounded.

\begin{theorem}\label{thm:peral} Let $1<p<\infty$, and fix $\lambda>0$.  
 The regularised operators 
$$(\lambda-\Delta)^{-{\alpha}/{2}}\cos(t\sqrt{-\Delta}), \quad t\in\R,$$ are bounded on 
$L^p(\R^d)$ if and only if
$\alpha \geq (d-1)|\frac1p - \frac12|$.
\end{theorem}

This result has been extended in many directions, and included in a general 
theory of Fourier integral operators (see, in particular, the celebrated paper 
by Seeger, Sogge, and Stein \cite{SSS}, and Section IX.5 of Stein's book \cite{Stein}). 
\medskip

Our paper is part of a long term programme (see also the Hardy space theory 
developed in \cite{MM,MMS12, MMS13} and \cite{MNP11,MNP12,Por}) to expand harmonic analysis 
of Ornstein-Uhlenbeck operators beyond Fourier multipliers and towards Fourier 
integral operators. We first remark that no analogue of Miyachi-Peral's result can hold 
in this context (see Theorem \ref{thm:reg-cos}). This can be seen as a consequence of the fact that, in $L^p$, 
$(e^{-t\calL})_{t \geq 0}$ only extends analytically to a sector of angle 
$\omega_{p}<\frac{\pi}{2}$ (except if $p=2$). This is related to the fact that 
there are no Sobolev embeddings in the Ornstein-Uhlenbeck context, and, in a 
sense, no non-holomorphic functional calculus in $L^p$ for $p \neq 2$ (see 
\cite{HMM}).

Perhaps surprisingly (given that our space of variables is not geometrically doubling), 
we can nonetheless establish the 
fundamental estimates that underpin spectral multiplier theory (see e.g. 
\cite{COSY} and the references therein), namely the finite speed of propagation 
of $(e^{it\D})_{t \in \R}$, and the $L^2$-$L^2$ off-diagonal bounds of 
Davies-Gaffney type for $(e^{-t\calL})_{t \geq 0}$. 
The former generalises to the Ornstein-Uhlenbeck context the following 
classical result for the wave group. Let
$D$ be the Dirac operator on $L^2(\R^d)\oplus L^2(\R^d;\C^d) = L^2(\R^d;\C^{d+1})$
defined by \eqref{eq:diracD}.

\begin{theorem}\label{thm:finite-speed-D}
 The $C_0$-group $(e^{itD})_{t\in\R}$ on $L^2(\R^d;\C^{d+1})$ has unit speed 
 of propagation, meaning that, if $f\in L^2(\R^d;\C^{d+1})$ is supported in a set $K$, then
 $e^{itD}f$ is supported in $\{x \in \R^{d}:\ \dist(x,K) \leq |t|\}$.
\end{theorem}

The $L^2$-$L^2$ off-diagonal estimates (which can be deduced from Theorem \ref{thm:finite-speed-D}) are integrated heat kernel bounds such as
$$
\|1_{G} e^{t\Delta} (1_{F}u)\|_{2} \lesssim \exp\big(\!-\!\frac{d(F,G)^{2}}{t}\big)\|u\|_{2},
$$
for $F,G\subseteq \R^{d}$, $u \in L^{2}(\R^d)$, and $t>0$.
These bounds play a key role in spectral multiplier theory, but hold far more generally than 
standard pointwise heat kernel bounds (which do not hold, in particular, for Ornstein-Uhlenbeck operators, even in finite dimension).

In a future project, we plan to use the off-diagonal estimates, together with the aforementioned Hardy space 
theory, to study perturbations of Ornstein-Uhlenbeck operators arising from 
non-linear stochastic PDE. 

\medskip
Let us now turn to a summary of the results of this paper.
After a brief introduction to Ornstein-Uhlenbeck operators $\L$
in an infinite-dimensional 
setting in Section \ref{sec:OU}, we begin in Section \ref{sec:groupL}
by proving analogues of Theorems \ref{thm:dirac} and 
\ref{thm:peral} for the operators $i\L$. This is
somewhat easier than proving analogues for $i\sqrt{-\L}$, which is
done in Section \ref{sec:groupD}. Roughly speaking, we find that both $i\L$ and $i\sqrt{-\L}$
generate groups in $L^p$ with respect to the invariant measure if and only if $p=2$ and $\L$ is 
self-adjoint. 
Moreover, in contrast with the Euclidean case, we show  that no amount of resolvent 
regularisation will push the groups into $L^p$. 

We turn to the analogue of Theorem \ref{thm:finite-speed-D} in Section \ref{sec:finite-speed} 
and prove that the group generated 
$i\D$ has finite speed of propagation, whereas the group generated by $i\L$ does not. 
To the best of our knowledge, the former is the first result of this kind in an infinite-dimensional setting.

In Section \ref{sec:DG}, we prove $L^2$-$L^2$ off-diagonal bounds for various operators
associated with $\L$,  such as $e^{t\L}$ and $\nabla_H e^{t\L}$, where $\nabla_H$ is a suitable directional gradient
introduced in Section \ref{sec:OU}.
In the symmetric case, this is done as an application of finite speed of 
propagation, and the off-diagonal bounds are of Gaffney-Davies type. 
In the non-symmetric case, we obtain off-diagonal bounds  for the resolvent operators $(I - t^2\L)^{-1}$ by a direct method.

\section{Non-symmetric Ornstein-Uhlenbeck operators}\label{sec:OU}

We begin by describing the setting that we will be using throughout the paper. 
We fix a real Banach space $E$ and a real Hilbert space $H$,
which is continuously embedded in $E$ by means on an inclusion operator $$i_H: H\embed E.$$ 
Identifying $H$ with its dual via the Riesz representation theorem,
we define $Q_H:= i_H\circ i_H\s$. Let $S = (S(t))_{t\ge 0}$ be a $C_0$-semigroup on $E$ with generator $A$. 

\begin{assumption}\label{ass:mu-infty}
 There exists a centred Gaussian Radon measure $\mu$ on $E$ whose covariance operator 
 $Q_\mu\in\calL(E\s,E)$ is given by 
$$\lb Q_\mu x\s, y\s\rb = \int_0^\infty \lb Q_H S(s)\s x\s, S(s)\s y\s\rb \,\ds, \quad x\s,y\s\in E\s,$$
the convergence of the integrals on the right-hand side being part of the assumption. 
\end{assumption}

The relevance of Assumption \ref{ass:mu-infty} is best explained in terms of its meaning in the 
context of stochastic evolution equations. For this we need some terminology.
Let $W_H$ be an {\em $H$-cylindrical Brownian motion} on an 
underlying probability space $(\Omega,\P)$. By definition, this means that 
$W_H$ is a bounded linear operator from $L^2(\R_+;H)$ to $L^2(\Omega)$ such that for all 
$f,g\in L^2(\R_+;H)$ the random variables $W_H(f)$ and $W_H(g)$ are centred Gaussian variables
and satisfy $$\E(W_H(f)W_H(g)) = \lb f,g\rb,$$ 
where $\lb f,g\rb$ denotes the inner product of $f$ and $g$ in $L^2(\R_+;H)$. 
The operators $W_H(t):H\to L^2(\Omega)$ defined by $W_H(t)h:= W_H(\one_{[0,t]}\otimes h)$ are then well defined, and 
for each $h\in H$ the family $(W(t)h)_{t\ge 0}$ is a Brownian motion; it is a standard Brownian motion if 
the vector $h$ has norm one.
Moreover, for orthogonal unit vectors $h_n$, the Brownian motions $(W(t)h_n)_{t\ge 0}$ are independent.  
For more information the reader is referred to \cite{Nee-Can}.

It is well known that Assumption \ref{ass:mu-infty} holds if and only if the linear stochastic evolution equation
\begin{equation}\label{eq:SCP}\tag{SCP} dU(t) = AU(t) + i_H \,dW_H(t), \quad t\ge 0,\end{equation}
is well-posed and admits an invariant measure. 
More precisely, under Assumption \ref{ass:mu-infty} the problem \eqref{eq:SCP} is well-posed and 
the measure $\mu$ is invariant,
and conversely if \eqref{eq:SCP} is well-posed and admits an invariant measure, then Assumption \ref{ass:mu-infty} 
holds and
the measure $\mu$ is invariant for \eqref{eq:SCP}.
In particular, if \eqref{eq:SCP} has a unique invariant measure,
it must be the measure $\mu$ whose existence is guaranteed by Assumption \ref{ass:mu-infty}. Details may be 
found in \cite{DaPZab, vNW06}, where also the rigorous definitions are provided 
for the notions of solution and invariant measure for \eqref{eq:SCP}. 

\begin{remark}
 More generally one may consider Ornstein-Uhlenbeck operators associated with the problem 
 \begin{equation}\label{eq:SCP2}\tag{SCP} dU(t) = AU(t) + \sigma \,dW_H(t), \quad t\ge 0,\end{equation}
where $\sigma: H \to E$ is a given bounded operator. This does not add any generality, however, 
as can be seen from the following reasoning. First, by the properties of the It\^o 
stochastic integral, replacing $H$ by $H\ominus\Null(\sigma)$ (the orthogonal complement
of the kernel of $\sigma$)  affects  neither the solution process $(U(t,x))_{t\ge 0}$ nor the invariant measure $\mu$,
and therefore this replacement 
leads to the same operator $\L$. Thus we may assume $\sigma$ to be injective. But once we have 
done that, we may identify $H$ with its image $\sigma(H)$ in $E$, 
which amounts to replacing $\sigma$ by the inclusion mapping $i_{\sigma(H)}$ of $\sigma(H)$ into $E$.
\end{remark}

In what follows, Assumption \ref{ass:mu-infty} will always be in force even if it is not explicitly mentioned.
Let $(U(t,x))_{t\ge 0}$ denote the solution of \eqref{eq:SCP} with initial value $x\in E$.
The formula
$$ P(t)f(x) := \E (f(U(t,x))),\quad t\ge 0, \ x\in E,$$
defines a semigroup of linear contractions $P = (e^{t\L})_{t\ge 0}$ 
on the space $B_{\rm b}(E)$ of bounded scalar-valued Borel functions on $E$, the so-called 
{\em Ornstein-Uhlenbeck semigroup} associated with the data $(A,H)$. By Jensen's inequality, 
this semigroup extends to a  $C_0$-semigroup of contractions on $L^p(E,\mu)$. 
Its generator will be denoted by $\L$, and henceforth we shall write $P(t) = e^{t\L}$ for all $t \geq 0$. 

In most of our results we will make the following assumption.

\begin{assumption}\label{ass:L-analytic}
For some (equivalently, for all) $1<p<\infty$ the semigroup $(e^{t\L})_{t\ge 0}$ extends to an analytic 
$C_0$-semigroup on $L^p(E,\mu)$. 
\end{assumption}

Here we should point out that, although the underlying spaces $E$ and $H$ are real, 
function spaces over $E$ will always be taken to be complex.
The independence of $p\in (1,\infty)$ is a consequence of the Stein interpolation theorem.

The problem of analyticity of $(e^{t\L})_{t\ge 0}$ has been studied by various authors in \cite{Fuh, Gol, GolNee, MN1}.
In these papers, various necessary and sufficient conditions for analyticity were obtained. Analyticity
always fails for $p=1$; this observation goes back to \cite{DavSim} where it was phrased for
the harmonic oscillator; the general case follows from \cite{CFMP, MN1}.

Under Assumption \ref{ass:L-analytic} it is possible to represent $\L$ in divergence form. 
For the precise statement of this result we need to introduce the following terminology.
A $C_{\rm b}^1$-{\em cylindrical function} is a function $f:E\to\R$ of the form
 \begin{equation}
 \label{eq:cyl}
 f(x) = \phi(\lb x,x_1\s\rb, \hdots,\lb x,x_n\s\rb)
 \end{equation}
 for some $n\ge 1$, with  $x_j\s\in E\s$ for
 all $j=1,\hdots,n$ and $\phi\in C_{\rm b}^1(\R^n)$. 
The {\em gradient in the direction of $H$} of such a function is defined by
\begin{align}\label{eq:DH} \nabla_H f(x) := \sum_{j=1}^n \frac{\partial \phi}{\partial x_j}(\lb
x,x_1\s\rb,\dots,\lb x,x_n\s\rb) \, i_H\s x_j\s, \quad x\in E.
\end{align}
If $(e^{t\L})_{t\ge 0}$ is analytic on $L^p(E,\mu)$ for some/all $1<p<\infty$, then $\nabla_H$ is
closable as a densely defined operator from $L^p(E,\mu)$ to $L^p(E,\mu;H)$
\cite[Proposition 8.7]{GolNee}. In what follows, $\nabla_H$ will 
always denote this closure and
 $\Dom_p(\nabla_H)$ and $\Ran_p(\nabla_H)$ denote its domain and range. For $p=2$ we usually 
omit the subscripts and write
$\Dom(\nabla_H)=\Dom_2(\nabla_H)$ and $\Ran(\nabla_H)=\Ran_2(\nabla_H)$.

It was shown in \cite{MN1} that if $(e^{t\L})_{t\ge 0}$ is 
analytic on $L^2(E,\mu)$, then $-\L$ admits the `gradient form' representation 
\begin{align}\label{eq:L} 
-\L = \frac12\nabla_H\s B \nabla_H 
\end{align}
for a unique bounded operator $B\in\calL(H)$ which satisfies $$B+B\s = 2I.$$ 
Note that this identity implies the coercivity estimate $\lb Bh,h\rb_H \ge \n h\n_H^2$ for all $h\in H$.

The rigorous interpretation of \eqref{eq:L} is that for $p=2$ the operator $-\L$ is the sectorial operator
associated with the sesquilinear form
$$ (f,g)\mapsto \frac12\lb B\nabla_H f, \nabla_H g\rb.$$
Therefore $\L$ generates an analytic $C_0$-semigroup of contractions on $L^2(E,\mu)$.

It is not hard to show (see \cite{GolNee}) that 
$$\hbox{$\L$ is self-adjoint on 
$L^2(E,\mu)$ if and only if $B = I_H$}$$
where $I_H$ is the identity operator on $H$.
In that case we have $\Dom(\sqrt{-\L}) = \Dom(\nabla_H)$ and 
\begin{align}\label{eq:Riesz-L2}
 \n \sqrt{-\L}f\n_2^2 = \frac12\n \nabla_H f\n_2^2. 
\end{align}

\begin{remark}Necessary and sufficient conditions for equivalence of homogeneous norms
$\n \sqrt{-\L}f\n_p \eqsim\n \nabla_H f\n_p$
in the non-symmetric case have been obtained in \cite{MN}, thereby unifying earlier 
results for the symmetric case in infinite dimensions \cite{CG01, Sh92} and 
the non-symmetric case in finite dimensions \cite{MPRS02}.
\end{remark}

\section{The $C_0$-group generated by $i\L$}\label{sec:groupL}

We start with an analogue of H\"ormander's theorem:

\begin{theorem}\label{thm:Hor-iL}
 Let Assumptions \ref{ass:mu-infty} and \ref{ass:L-analytic} hold and let $1\le p<\infty$.
 The operator $i\L$ generates a $C_0$-group on $L^p(E,\mu)$ if and only if $p=2$ and $\L$ is self-adjoint
 on $L^2(E,\mu)$.
\end{theorem}

\begin{proof}
Let $1<p<\infty$ be fixed and suppose that 
$i\L$ generates a $C_0$-group on $L^p(E,\mu)$. Then, by \cite[Corollary 3.9.10]{ABHN}, the 
semigroup $(e^{t\L})_{t\ge 0}$ on $L^p(E,\mu)$ generated by $\L$ is analytic of 
angle $\pi/2$ and the group generated by $i\L$ is 
its boundary group, i.e.,  $$e^{it\L}f = \lim_{s\downarrow 0} e^{i(s+it)\L}f$$ for all $f\in L^p(E,\mu)$.
But it is well known \cite{CFMP, MN1} that $(e^{t\L})_{t\ge 0}$ fails to be analytic on $L^1(E,\mu)$ and
that for $1<p<\infty$ the optimal angle of analyticity $\theta_p$ of $(e^{t\L})_{t\ge 0}$ in $L^p(E,\mu)$
is given by 
\begin{align}\label{eq:angle}\cot \theta_p :=
\frac{\sqrt{(p-2)^2+p^2\|B-B\s\|^2}}{2\sqrt{p-1}},
\end{align}
with $B\in \calL(H)$ the operator appearing in \eqref{eq:L}. 
If either $p\not=2$ or $B\not=B\s$, this angle is strictly less than $\pi/2$.
\end{proof}

\begin{remark} 
An alternative proof of self-adjointness can be given that does not rely on the formula
\eqref{eq:angle} for the optimal angle.
It relies on the following result on numerical ranges. If
 $G$ is the generator of a $C_0$-semigroup on a complex Hilbert space $\mathscr{H}$ such that 
$\lb Gx,x\rb \in\R$ for all $x\in\Dom(G)$, then $G$ is self-adjoint.
Indeed, for any $\lambda\in\R$ the operator $\lambda-G$ has real numerical range. Therefore, for any
real $\lambda\in \varrho(G)$
the resolvent operator $R(\lambda,G)$ has real numerical range. Hence, 
by \cite[Theorem 1.2-2]{GusRao}, $R(\lambda,G)$ is self-adjoint, and then the same is true for $G$.

Now let us revisit the proof of self-adjointness in the theorem for $p=2$. 
By second quantisation \cite{ChojGol96, GolNee}, the analytic semigroup 
generated by $\L$  on $L^2(E,\mu)$ is 
contractive in the right half-plane $\{z\in\C: \ \Re z>0\}$. 
By general semigroup theory (see, e.g., \cite[Proposition 7.1.1]{Haase}), this implies that the numerical range 
of $\L$ is contained in $(-\infty,0]$. By the observation just made,
this implies that $\L$ is self-adjoint on $L^2(E,\mu)$. 
\end{remark}

Not only does $i\L$ fail to generate a $C_0$-group on $L^p(E,\mu)$ unless $p=2$ and $\L$ is self-adjoint, but
the situation is in fact worse than that. As we will see shortly, for any given $\lambda>0$ and $\a>0$,
the regularised operators $$(\lambda-\L)^{-\a}e^{it\L}$$ 
fail to extend to bounded operators on $L^p(E,\mu)$, unless $p=2$ and $\L$ is self-adjoint. 
This result contrasts with the 
analogous situation for the Laplace operator: it is a classical result of Lanconelli \cite{Lan} (see also Da Prato and Giusti \cite{daPG} for integer values of $\alpha$) that  the regularised Schr\"odinger operators $(\lambda-\Delta)^{-\alpha}e^{it\Delta}$ are bounded
on $L^p(\R^d)$ for all $\alpha> n|\frac1p-\frac12| $.
\medskip

With regard to the rigorous statement of our result 
there is a small issue here in the non-self-adjoint case, for then it is not
even clear how to define these operators for $p=2$.
We get around this in the following way. Any reasonable definition should respect the identity
$$e^{s\L}[(\lambda-\L)^{-\alpha} e^{it\L}] = (\lambda-\L)^{-\alpha}e^{(s+it)\L}, \quad s>0.$$
More precisely, it should be true that the mapping $z\mapsto (\lambda-\L)^{-\alpha}e^{z\L}$ is holomorphic 
in $\{\Re z>0\}$ and that the above identity holds. In the converse direction,
if the mapping $z\mapsto (\lambda-\L)^{-\alpha}e^{z\L}$ (which is well-defined and holomorphic on
an open sector about the positive real axis) extends holomorphically to a function $F_\alpha$ 
on $\{\Re z>0\}$ which is
bounded on every bounded subset of this half-plane, then by general principles the strong
non-tangential limits $ \lim_{s\downarrow 0}  F_\a(s+it)$ exist for almost all $t\in\R$.
For these $t$ we may define the operators 
$ (\lambda-\L)^{-\alpha}e^{it\L}$ to be this limit. In what follows, ``boundedness of the operators
$(\lambda-\L)^{-\alpha}e^{it\L}$ in $L^p(E,\mu)$'' will always be understood in this sense. 

This procedure defines the operators for almost all $t\in \R$. As a side-remark we mention that 
this can be improved by using a version 
of the argument in \cite[Proposition 9.16.5]{ABHN}. For $\beta\ge \a$ let 
$G_\beta$ be the set of full measure for which the non-tangential strong 
limits $ \lim_{s\downarrow 0}  F_\a(s+it)$ exist.
We claim that $G_{\b} = \R$ for all $\b\ge 2\a$.
To prove this, first observe that for all $\gamma'>\gamma\ge\a$ we have 
$G_\gamma\subseteq G_{\gamma'}$ and $G_\gamma+G_{\gamma'} \subseteq G_{\gamma+\gamma'}$. 
If the claim were wrong, then there would be a 
$t\in \complement G_\b$ for some $\b\ge 2\a$. But then for any $t'\in G_{\frac12\beta}$ we have 
$t-t'\in \complement G_{\frac12\beta}$,
for otherwise the identity $t =  t' + (t-t')$ implies $t \in G_{\frac12\beta} + G_{\frac12\beta}\subseteq G_{\beta}$.
This contradiction concludes the proof of the claim.

\begin{theorem}\label{thm:regularised}
 Let Assumptions \ref{ass:mu-infty} and \ref{ass:L-analytic} hold and let $1<p<\infty$.
 If, for some $\lambda>0$ and $\a>0$, the operators $(\lambda-\L)^{-\a}e^{it\L}$, $t\in\R$, 
 are bounded in $L^p(E,\mu)$, then $p=2$ and $\L$ is self-adjoint. 
\end{theorem}

\begin{proof} For all $s>0$,  
the operators $(\lambda-\L)^{\alpha}e^{s\L}$ are bounded in $L^p(E,\mu)$ 
by the analyticity of the semigroup $(e^{t\L})_{t\ge 0}$.
 The assumptions of the theorem then imply that the operators
 \begin{align}\label{eq:hol-ext}
 e^{(s+it)\L} = (\lambda-\L)^{\alpha}e^{s\L} \circ (\lambda-\L)^{-\alpha}e^{it\L}
 \end{align}
 are bounded on $L^p(E,\mu)$ for all $s>0$ and $t\in\R$, in the sense that the right-hand side provides
 us with an analytic extension of $t\mapsto e^{t\L}$ to $\{\Re z>0\}$.  
 But, as was observed in the proof of Theorem \ref{thm:Hor-iL},
 for $p\not=2$ and $B\not= B\s$ the optimal angle of holomorphy of this semigroup   
 is strictly smaller than $\pi/2$.
 \end{proof}

 \begin{remark}\label{rem:expreg} The `exponentially regularised' operators 
  $e^{s \L}e^{it\L}$ extend to $L^p(E,\mu)$ if $s+it$ belongs to the connected component of the  
  domain of analyticity in $L^p(E,\mu)$ of $z\mapsto e^{z\L}$ which 
  contains the positive real axis. For the standard Ornstein-Uhlenbeck operator
  in finite dimensions (see \eqref{eq:classOU} for its definition), this is the Epperson region 
  $$E_p = \{x+iy\in\C: \ |\sin y| \le \tan \theta_p \sinh x\},$$ where $\theta_p = \arccos|2/p-1|$
  \cite[Theorem 3.1]{Epp}  (see also \cite[Proposition 1.1]{GC}). It contains the right-half plane
  $\{z\in\C: \ \Re z> s_p\}$ for a suitable abscissa $s_p>0$.
  Hence, for all $s>s_p$ the operators $e^{s \L}e^{it\L}$, $t\in\R$, extend to $L^p(E,\mu)$.
  
  In the general case, a similar conclusion can be drawn in the 
  presence of hypercontractivity (which holds if Assumption \ref{ass:L-spectral-gap} below is satisfied, 
  see \cite{CG02}). In that case the operators $e^{s \L}e^{it\L}$
  are bounded on $L^p(E,\infty)$ for all $s>s_p^*$ and $t\in\R$, where $s_p^*>0$ is the infimum of all $s>0$ 
  with the property that $e^{s\L}$ maps $L^{p}(E,\mu)$ into $L^{2}(E,\mu)$ (if $1<p<2)$, 
  respectively $L^{2}(E,\mu)$ into $L^{p}(E,\mu)$ (if $2<p<\infty)$.
 \end{remark}

\section{The $C_0$-groups generated by $i\sqrt{-\L}$ and $i\D$}\label{sec:groupD}

Throughout this section, Assumptions \ref{ass:mu-infty} and 
\ref{ass:L-analytic} are in force. 
On the direct sum $L^p(E,\mu)\oplus L^p(E,\mu;H)$, $1<p<\infty$, we introduce the {\em Hodge-Dirac operator}
\begin{equation}\label{eq:Dirac}
\D := \bma 0 & \nabla_H\s B \\ \nabla_H & 0\ema.
\end{equation}
Hodge-Dirac operators have their origins in Dirac's desire to use first-order operators 
that square to the Laplacian.  They  are commonly used in Riemannian geometry, where they 
arise as $d+d^{*}$ for the exterior derivative $d$.
In their influential paper \cite{AKM}, Axelsson, Keith, and McIntosh have introduced a 
general operator theoretic framework that allows one to transfer ideas used in geometry 
to problems in harmonic analysis and PDE related to Riesz transform estimates. 
For Ornstein-Uhlenbeck operators, this perspective has been introduced in \cite{MN}.

On various occasions we will use the fact (see \cite{AKM}) that $\D$
is bisectorial on $L^2(E,\mu)\oplus L^2(E,\mu;H)$. We recall that a
closed operator $A$ is called {\em bisectorial} if
$i\R\setminus\{0\}\subseteq \varrho(A)$ and 
$$\sup_{t\not=0} \n (I+itA)^{-1}\n <\infty.$$
For some background on bisectoriality we recommend the lecture notes \cite{ADM} and Duelli's
Ph.D. thesis \cite{Due05}. 

Note the formal identity
$$\tfrac12\D ^2 = \tfrac12  \bma -\nabla_H\s B \nabla_H & 0 \\ 0 &- \nabla_H\nabla_H\s B \ema= \bma -\L & 0 \\ 0 &- \uL \ema.$$
Here, the operator $\uL = -\frac12\nabla_H \nabla_H\s B$ is defined as follows. 
First, we define $\nabla_H\nabla_H\s$ on $L^2(E,\mu;H)$ by means of the 
form $(u,v)\mapsto \lb \nabla_H\s u,\nabla_H\s v\rb$, and use this operator to 
define $ -\frac12\nabla_H \nabla_H\s B$ in the natural way on the 
domain $\Dom(\uL) = \{u\in L^2(E,\mu;H): \ Bu\in \Dom(\nabla_H\nabla_H\s)\}$. 
The operator $\uL$ generates a bounded analytic $C_0$-semigroup on $L^2(E,\mu;H)$ and we have
$$e^{t\uL}\nabla_H = \nabla_H e^{t\L}. $$
This identity implies that $(e^{t\uL})_{t\ge 0}$ restricts to a bounded analytic $C_0$-semigroup on
$\ov{\Ran(\nabla_H)}$.

The situation for $1<p<\infty$ is slightly more subtle. 
The semigroup $(e^{t\uL})_{t\ge 0}$ on $\ov{\Ran(\nabla_H)}$ can be shown to extend 
to a bounded analytic $C_0$-semigroup on $\ov{\Ran_p(\nabla_H)}$. We then define $\uL$ on $\ov{\Ran_p(\nabla_H)}$ 
as its generator. This suggests to consider the part of the Dirac operator $\D$ 
in $L^p(E,\mu)\oplus \ov{\Ran(\nabla_H)}$,
and indeed it can be shown that this operator is bisectorial 
on $L^p(E,\mu)\oplus \ov{\Ran(\nabla_H)}$. 
The reader is referred to \cite{MN} for the details. 
If $\uL$ has a bounded $H^\infty$-calculus on $\ov{\Ran_p(\nabla_H)}$ (this is the case 
if $E = H = \R^d$ and also if $\L$ is self-adjoint on $L^2(E,\mu)$), then it follows from the 
second part of \cite[Theorem 2.5]{MN} that $\D$ is bisectorial on all of  $L^p(E,\mu)\oplus L^p(E,\mu;H)$.

If $\D $ is self-adjoint
on the direct sum $L^2(E,\mu)\oplus L^2(E,\mu;H)$, then
$i\D $ generates a bounded $C_0$-group on this space by Stone's theorem.
In the non-self-adjoint case, one may ask whether 
it is still true that $i\D $ generates a $C_0$-group on $L^p(E,\mu)\oplus 
L^p(E,\mu;H)$ for certain exponents $1<p<\infty$. In the light of the above discussion we
have to be a little cautious as to the precise meaning of this 
question; we ask whether the restriction of $(e^{it\D})_{t\in\R}$ to 
$[L^2(E,\mu)\oplus L^2(E,\mu;H)] \cap [L^p(E,\mu)\oplus L^p(E,\mu;H)]$
extends to a $C_0$-group on $L^p(E,\mu)\oplus 
L^p(E,\mu;H)$. Alternatively, one may ask whether 
$i\D$ generates a $C_0$-group on $L^p(E,\mu)\oplus \ov{\Ran(\nabla_H)}$.
In this formulation of the question one may interpret $\D$ as the bisectorial operator
on $L^p(E,\mu)\oplus \ov{\Ran(\nabla_H)}$ as outlined above.

 In  the one-dimensional Euclidean situation, $(e^{itD})_{t \in \R}$ can be 
expressed in terms of the translation group.  This  suggests that the answer to both questions 
for $\D$ could be positive  at least in dimension one. The following result shows
however that the answer is always negative, except when  
$p=2$ and $\L$ is self-adjoint.

\begin{theorem}\label{thm:main}
Let Assumptions  \ref{ass:mu-infty} and \ref{ass:L-analytic} hold and let $1<p<\infty$. 
The following assertions are equivalent:
\begin{enumerate}[\rm(i)]
 \item the operator $i\D$ generates a $C_0$-group on $L^p(E,\mu)\oplus L^p(E,\mu;H)$;
 \item the operator $i\D$ generates a $C_0$-group on $L^p(E,\mu)\oplus \overline{\Ran_p(\nabla_{H})}$;
 \item the operator $i\sqrt{-\L}$ generates a $C_0$-group on $L^p(E,\mu)$;
 \item the operator $\L$ generates a $C_0$-cosine family on $L^p(E,\mu)$;        
 \item $p=2$ and $\calL$ is self-adjoint on $L^2(E,\mu)$. 
\end{enumerate}
\end{theorem}

A thorough discussion of cosine families is presented 
in \cite{ABHN}, which will serve as our standard reference.
For the reader's convenience we recall some relevant definitions.
Let $X$ be a Banach space. A strongly continuous function $C:\R \to \calL(X)$ is called a
{\em $C_0$-cosine family} if $C(0)=I$ and
$$2C(t)C(s)=C(t+s)+C(t-s), \quad t,s\in\R.$$
By an application of the uniform boundedness theorem, $C_0$-cosine functions are exponentially
bounded; see \cite[Lemma 3.14.3]{ABHN}. Denoting the exponential type of $C$ by $\omega$, 
by \cite[Proposition 3.14.4]{ABHN} 
there exists a unique closed densely defined
operator $A$ on $X$ such that for all $\lambda>\omega$ we have $\lambda^2\in\varrho(A)$ and  
\begin{align}\label{eq:genC}
\lambda (\lambda^2-A)^{-1}x = \int_0^\infty e^{-\lambda t} C(t)x\,\dt, \quad x\in X. 
\end{align}
This operator $A$ is called the {\em generator} of $C$.

\begin{proof}[Proof of Theorem \ref{thm:main}]
(i)$\Rightarrow$(v) and (ii)$\Rightarrow$(v): \ 
By a well-known result from semigroup theory, if $\A$ generates a  
$C_0$-group $G$ on a Banach space $X$, then $\A^2$ generates an analytic $C_0$-semigroup $T$ 
of angle $\frac12\pi$ on $X$ given by the formula 
\begin{align}\label{eq:Tt} T(z)x = \frac1{\sqrt{2\pi z}} 
\int_{-\infty}^\infty e^{-t^2/4z}G(t)x\,\dt, \quad \Re z>0.
\end{align}

Suppose now that (i) or (ii) holds.
By the observation just made $-\D^2$ generates an analytic $C_0$-semigroup on
$L^p(E,\mu)\oplus L^p(E,\mu;H)$, respectively on $L^p(E,\mu)\oplus \ov{\Ran_p(\nabla_H)}$, of angle $\frac12\pi$.
In particular, by considering the first coordinate, $\L$ generates an analytic $C_0$-semigroup on $L^p(E,\mu)$ 
of angle $\frac12\pi$. As we have seen in the proof of Theorem \ref{thm:Hor-iL}, this implies that $p=2$
and that $\L$ is self-adjoint.

\smallskip
(v)$\Rightarrow$(i) and (v)$\Rightarrow$(ii): \ For $p=2$, the self-adjointness of $\L$ 
implies $B=I_H$ and $\uL = \nabla_H\nabla_H\s$, 
and therefore the realisations of $\D$ considered in (i) and (ii) are both self-adjoint. 
Now (i) and (ii) follows from Stone's theorem.

\smallskip
(v)$\Rightarrow$(iii) and (v)$\Rightarrow$(iv): \ The group and cosine family may be defined through 
the Borel functional calculus of $-\L$ 
by $e^{i\sqrt{-\L}}$ and $\cos(t\sqrt{-\L})$; it follows from \eqref{eq:genC} that $\L$ is the  generator
of this cosine family.

\smallskip
(iv)$\Rightarrow$(iii)$\Rightarrow$(v): \ 
By a theorem of Fattorini \cite{Fat69} (see also \cite[Theorem 3.16.7]{ABHN})
the operator $i\sqrt{-\L}$ generates a $C_0$-group on $L^p(E,\mu)$.
Then by \eqref{eq:Tt}, its square $\L$ generates an analytic $C_0$-semigroup 
on $L^p(E,\mu)$ of angle $\pi/2$, and we have already seen that this forces $p=2$ and 
self-adjointness of $\L$.
\end{proof}

Let $L_0^p(E,\mu)$ be the codimension-one subspace of $L^p(E,\mu)$
comprised of all functions $f$ for which $\overline f := \int_E f\,{\rm d}\mu = 0$.

\begin{lemma}\label{lem:null} Let Assumptions \ref{ass:mu-infty} and \ref{ass:L-analytic}
hold
and let $1<p<\infty$. Then
\begin{align*}
\Null_p(\L) & = \Null_p(\sqrt{-\L}) = \Null_p(\nabla_H) = \C\one, \\ 
\ov{\Ran_p(\L)} & = \ov{\Ran_p(\sqrt{-\L})} = \ov{\Ran_p(\nabla_H\s B)}  = L_0^p(E,\mu). 
\end{align*}
On $\ov{\Ran_p(\nabla_{H})}$ we have
\begin{align*}
\Null_p(\uL) & =\Null_p(\sqrt{-\uL}) =  \Null_p(\nabla_H\s B) = \{0\}, \\  \ov{\Ran_p(\uL)} 
& = \ov{\Ran_p(\sqrt{-\uL})} =  \ov{\Ran_p(\nabla_H)}. 
\end{align*}
\end{lemma}

   \begin{proof}
All this is contained in \cite[Proposition 9.5]{MN}, with the exception of the identities
$\ov{\Ran_p(\L)} = L_0^p(E,\mu)$ and the four equalities relating the kernels and closed ranges of $\L$ and $\uL$
with those of their square roots.

Since $\L$ is sectorial we have a direct sum decomposition $L^p(E,\mu) = \Null(\L)\oplus \ov{\Ran(\L)}
= \C\one \oplus \ov{\Ran(\L)}$. If $f$ is any $C_{\rm b}^1$-cylindrical function belonging to
$\Dom(\L)$, then $\lb \L f,\one\rb = \lb B\nabla_H f,\nabla_H \one\rb = 0$.
Since these functions $f$ are dense in $\Dom(\nabla_H)$, and $\Dom(\nabla_H)$ is dense in $\Dom(\L)$,
it follows from $\lb \L f,\one\rb = 0$ that $\Ran(\L) \subseteq L_0^p(E,\mu)$.
Since both $\ov{\Ran(\L)}$ and $L_0^p(E,\mu)$ have codimension one, these spaces must in fact be equal.

The four equalities for the square roots follow from the general fact that if $S$ is sectorial or bisectorial
and $S^2$ is sectorial, then $\Null(S) = \Null(S^2)$ and $\ov{\Ran(S)} = \ov{\Ran(S^2)}$.
  \end{proof}

\begin{remark}
Assumption \ref{ass:mu-infty} implies the identity $$\lb \L f,g\rb + \lb f,\L g\rb 
= -\frac12 \int_E \lb \nabla_H f, \nabla_H g\rb_H\,{\rm d}\mu.$$
This establishes a connection with the theory of Dirichlet forms,
and part of the above lemma
could be deduced from it. A comprehensive treatment of this theory and its many ramifications 
is presented in the monograph  \cite{BGL}.  \end{remark}

 In the remainder of this section  we shall assume that $p=2$ and that $\L$ is self-adjoint, and 
turn to the problem of representing the group
generated by $i\L$ in an explicit matrix form.
Since $\sqrt{-\L}$ is self-adjoint,  $i\sqrt{-\L}$ generates a unitary 
$C_0$-group on $L^2(E,\mu)$ by Stone's theorem. By the Borel functional calculus for self-adjoint operators, 
we have the identities
\begin{align*}
C(t):= \cos(t\sqrt{-\L}) & = \frac12 (e^{it\sqrt{-\L}} + e^{-it\sqrt{-\L}}), \\ 
S(t):= \sin(t\sqrt{-\L}) & = \frac1{2i} (e^{it\sqrt{-\L}} - e^{-it\sqrt{-\L}}). 
\end{align*}

\begin{lemma}\label{lem:ucosine}
For all $t\in\R$ the formulas
$$\begin{aligned}
\uC(t) (\uD f):= \uD \Cos(t)f,  \\ 
\uS(t) (\uD f):= \uD \Sin(t)f,
\end{aligned}
 \qquad f\in\Dom(\uD),$$
define bounded operators $\uC(t)$ and $\uS(t)$ on
$\overline{\Ran(\nabla_H)}$ of norms
$\n \uC(t)\n \le \n C(t)\n$ and $\n \uS(t)\n \le \n S(t)\n$.
\end{lemma}
\begin{proof}
We will prove the statements for the cosines; the same proof works for the sines. 
In fact, all we use is that the operators $C(t)$ and $S(t)$ are bounded, 
map the constant function $\one$ to itself, and commute with $\L$.

First note that the operators $\uC(t)$ are well-defined on the range of $\nabla_H$. 
Indeed, if $\nabla_H  f = 0$, then 
$f = \overline f \one \in \C\one$ by Lemma \ref{lem:null}, where $\overline f = \int_E f\,{\rm d}\mu$. 
But $\Cos(t)\one=\one$ and therefore $\nabla_H\Cos(t)f = \overline f \nabla_H\Cos(t)\one 
= \overline f\nabla_H \one = 0$. 

From the representation  $\L = -\frac12 \nabla_H\s \nabla_H$ we have 
$\Dom(\sqrt{-\L }) = \Dom(\nabla_H )$ and
$ \n \sqrt{-\L} f \n_2 = \frac1{\sqrt{2}} \n \nabla_H  f\n_2$ (see \eqref{eq:Riesz-L2}).
This gives, for $f\in \Dom(\sqrt{-\L }) = \Dom(\nabla_H )$,
\begin{align*}\n \uC(t) \uD f\n_2
& = \n \nabla_H  \Cos(t)f \n_2 = \sqrt 2 \n \sqrt{-\L } \Cos(t)f \n_2
 \\ & = \sqrt 2 \n \Cos(t) \sqrt{-\L } f \n_2 
\le  \sqrt 2 \n\Cos(t)\n \n \sqrt{-\L } f \n_2= \n \Cos(t)\n\n \nabla_H  f \n_2.
\end{align*}
\end{proof}

Via the $H^\infty$-functional calculus of the self-adjoint bisectorial
operator $\D$ on $L^2(E,\mu)\oplus \ov{\Ran(\nabla_H)}$  
(see \cite{AKM, MN}) we can define the bounded operator ${\rm \sgn}(\D)$
on $L^2(E,\mu)\oplus \ov{\Ran(\nabla_H)}$.
The fact that this operator encodes Riesz transforms gives the main motivation of \cite{AKM}:
to obtain functional calculus results for second-order differential operators together 
with the corresponding Riesz transforms estimates through the functional calculus of an 
appropriate first-order differential operator. We recall the link between ${\rm \sgn}(\D)$
and Riesz transform in the next lemma.
The constant $1/\sqrt 2$ arising here is an artefact of the fact that
we consider the operator $-\L = \frac12\nabla_H\s\nabla_H$ (rather than
$\nabla_H\s\nabla_H$).

\begin{lemma}
 On $L^2(E,\mu)\oplus \ov{\Ran(\nabla_H)}$  we have $${\rm \sgn}(\D) = \frac1{\sqrt{2}}\bma 0 & \uR \\ R & 0\ema,$$
 where
\begin{align}\label{eq:Riesz} 
R :\sqrt{-\L}f & \mapsto \nabla_H f \ \ \hbox{and} \ \ R: \one\to 0, \\ 
\uR : \sqrt{-\uL}g & \mapsto  \nabla_H\s g,
\end{align}
denote the Riesz transforms associated with $-\L$ and $-\uL $, respectively. 
\end{lemma}

\begin{proof}
 Recall from Lemma \ref{lem:null} that $L^2(E,\mu) = \ov{\Ran(\sqrt{-\L})}\oplus \C\one$
and $\ov{\Ran(\nabla_H)} = \ov{\Ran(\sqrt{-\uL})}$.
Hence the above relations define $R$ and $\uR$ uniquely.

By the convergence lemma for the $H^\infty$-calculus we have
${\rm sgn}(\D) = \limn f_n(\D)$ strongly, where, for all $z \not \in i\R$, 
\begin{align*} 
f_n(z) =  \frac{nz}{1+n\sqrt{z^2}}.
\end{align*}
Here we take the branch of the square root that is holomorphic on $\C \setminus (-\infty,0]$.
Hence,
\begin{align*}
 {\rm sgn}(\D) & = \limn n\D (I+n\sqrt{\D^2})^{-1} 
 \\ &  = \limn \D \bma(n^{-1}+\sqrt{-2\L})^{-1} & 0 \\ 0 & (n^{-1}+\sqrt{2\uL})^{-1} \ema 
 \\ &  = \limn \bma 0 & \nabla_H\s(n^{-1}+\sqrt{-2\uL})^{-1} \\ \nabla_H (n^{-1}+\sqrt{-2\L})^{-1} & 0\ema. 
\end{align*}
It is immediate from the above representation that  ${\rm sgn}(\D) \bma \one\\ 0\ema = 0$.  
Also, 
\begin{align*}
 {\rm sgn}(\D) \bma \sqrt{-\L} f \\ \sqrt{-\uL} g \ema 
 & = \limn \bma \nabla_H\s(n^{-1}+\sqrt{-2\uL})^{-1} \sqrt{-\uL} g  \\ \nabla_H (n^{-1}+\sqrt{-2\L})^{-1} \sqrt{-\L} f \ema
  = \frac1{\sqrt{2}}\bma \nabla_H\s g \\ \nabla_H f\ema. 
\end{align*}
\end{proof}

\begin{lemma}\label{lem:Rcommute}
For all $t\in\R$ we have 
$$R\Cos(t) = \uC(t)R, \quad \Cos(t)\uR = \uR \uC(t),$$
$$R\Sin(t) = \uSin(t)R, \quad \Sin(t)\uR = \uR \uSin(t). $$ 
Furthermore, if $f\in\Dom(\sqrt{-\L })$, then $Rf\in\Dom(\sqrt{-\uL })$ and 
$$\sqrt{ -\uL}R f= R\sqrt{-\L }f . $$ 
Likewise, if $g\in\Dom(\sqrt{-\uL })$, then $\uR g\in\Dom(\sqrt{-\L })$ and 
$$ \sqrt{-\L }\uR g = \uR \sqrt{ -\uL}g .$$
Finally,
$$\frac12 \uR R = P_{\ov{\Ran(\nabla_H\s)}}= P_{L_0^2(E,\mu)}, \quad \frac12 R\uR  =  P_{\ov{\Ran(\nabla_H)}},$$ 
where the right-hand sides denote the orthogonal projections onto the indicated subspaces. 
\end{lemma}

 \begin{proof}
We have, 
for $f\in \Dom(\sqrt{-\L }) = \Dom(\nabla_H )$, 
\begin{align*}
R\Cos(t)\sqrt{-\L }f  = R\sqrt{-\L }\Cos(t) f = \nabla_H  \Cos(t)f =  \uC(t)\nabla_H  f = \uC(t)R \sqrt{-\L }f.
\end{align*}
This gives the first identity on the range of $\sqrt{-\L }$. On $\Null(\sqrt{-\L}) = \C\one$ 
(see Lemma \ref{lem:null}) 
the identity is trivial since $R \one =0$. 
Since $L^2(E,\mu) = \Null(\sqrt{-\L})\oplus \ov{\Ran(\sqrt{-\L})}$ by the sectoriality of $\sqrt{-\L}$,
this proves the first identity. The corresponding identity for the sine function is proved similarly.

The identities $R \sqrt{-\L} = \sqrt{-\uL }R$ and $\uR \sqrt{-\uL} = \sqrt{-\L }\uR$
follow by differentiating the identities
$\uSin(t)R = R\Sin(t)$ and $\uR \uSin(t) = \Sin(t)\uR$ at $t=0$. 

If $\nabla_H  f\in \Dom(\sqrt{-\uL}) = \Dom(\nabla_H \s)$, then
\begin{align*}
 \Cos(t)\uR \sqrt{-\uL} \nabla_H  f & = \Cos(t) \nabla_H \s \nabla_H  f =  -2\Cos(t) \L f = -2\L\Cos(t) f
\\ & = \nabla_H \s \uC(t) \nabla_H  f =  \uR\sqrt{-\uL}\uC(t)\nabla_H  f = \uR\uC(t)\sqrt{-\uL}\nabla_H  f.
\end{align*}
Noting that $\Ran(\uD)\cap \Dom(\sqrt{-\uL})$ is a core for $\Dom(\sqrt{-\uL})$ (it contains the $e^{t\uL}$-invariant 
dense linear subspace $\{e^{t\uL} \nabla_H  f: \ t>0, \ f\in \Dom(\nabla_H )\}$), 
it follows that $\Cos(t)\uR \sqrt{-\uL} = \uR\uC(t)\sqrt{-\uL}$. 
This proves the second identity on the range of $\sqrt{-\uL}$. 
Since $\ov{\Ran(\sqrt{-\uL})} = \ov{\Ran(\nabla_H)}$ this proves the identity  $\Cos(t)\uR  = \uR\underline{\Cos}(t)$. 
The corresponding sine identity is proved in the same way. 

Finally, the last two identities follow from
$$\bma P_{\ov{\Ran(\nabla_H\s)}} & 0 \\ 0 & P_{\ov{\Ran(\nabla_H)}}\ema =
P_{\ov{\Ran(\D)}} = {\rm sgn}^2(\D) = \frac12\bma 0 & \uR \\ R & 0\ema^2 =  \bma \frac12\uR R & 0 \\ 0 & \frac12 R\uR\ema,$$
recalling that $\Ran(\nabla_H\s) = L_0^2(E,\mu)$.
\end{proof}
 
\begin{theorem}\label{thm:main2} 
Let Assumption  \ref{ass:mu-infty} hold and suppose that $\L$ is self-adjoint on $L^2(E,\mu)$.
The $C_0$-group generated by $\frac{i}{\sqrt{2}}\D$ on $L^2(E,\mu)\oplus \overline{\Ran(\nabla_H)}$ is given by
\begin{equation}\label{eq:group} e^{\frac{i}{\sqrt{2}}t\D}= 
\bma \Cos(t) & \frac{i}{\sqrt{2}}\uR\uSin(t) \\ \frac{i}{\sqrt{2}}R \Sin(t) & \uC(t)\\  \ema, \qquad t\in\R.
\end{equation}
\end{theorem}
 By a scaling argument, this also gives a matrix representation for the group generated by $i\D$. 
\begin{proof}
On $L_0^2(E,\mu)\oplus\ov{\Ran(\nabla_H)}$ 
the group property follows by an easy computation using the lemmas and the addition formulas for 
 $C(t)$ and $S(t)$ and their underscored relatives  (see \cite[Formula (3.95)]{ABHN}).
On $\C\one\oplus \ov{\Ran(\nabla_H)}$ we argue similarly.

Strong continuity and uniform boundedness are evident from the corresponding properties of the matrix entries.
To see that its generator equals $i\D$, we set $G(t):= \bma \Cos(t) & \frac{i}{\sqrt{2}}\uR\uSin(t) \\ \frac{i}{\sqrt{2}}R \Sin(t) & \uC(t)\\  \ema$
 take $f\in \Dom(\L)$, $g\in \Dom(\uL)$, and differentiate. 
Lemma \ref{lem:Rcommute} then give us
\begin{align*} \lim_{t\downarrow 0} \frac1t\Big(G(t)\bma f\\ g\ema-\bma f\\ g\ema\Big) 
& = \frac{1}{\sqrt{2}}\bma 0 & i\uR\sqrt{-\uL } \\ iR \sqrt {-\L} & 0 \ema \bma f\\ g\ema 
\\ & =\frac{i}{\sqrt{2}}\bma 0 & \nabla_H \s \\ \nabla_H   & 0 \ema \bma f\\ g\ema
= \frac{i}{\sqrt{2}} \D  \bma f\\ g\ema. 
\end{align*}
This shows that $\frac{i}{\sqrt{2}}\D $ is an extension of the generator $\mathscr{G}$ of $(G(t))_{t\in\R}$. 
However, by the general theory of 
cosine families, the left-hand side limit exists  
if and only if $f\in \Dom(\sqrt {-\L})$ and $g\in \Dom(\sqrt{-\uL})$. In view of the domain identifications
$   \Dom(\sqrt {-\L}) = \Dom(\nabla_H )$ and $\Dom(\sqrt {-\uL}) = \Dom(\nabla_H \s)$ 
this precisely happens if and only if 
$\bma f \\ g \ema \in \Dom(\D )$. 
Therefore we actually have equality $\mathscr{G} = \frac{i}{\sqrt{2}}\D $.   
\end{proof}

We proceed with an analogue of Theorem \ref{thm:regularised}.

\begin{theorem}\label{thm:reg-cos}
 Let Assumption \ref{ass:mu-infty} hold and let $1<p<\infty$.
 If, for some $\lambda>0$ and $\a>0$, the operators $$(\lambda-\L)^{-\a}\cos(t\sqrt{-\L}), \quad t\in\R,$$ 
 extend to bounded operators on $L^p(E,\mu)$, then $p=2$ and $\L$ is self-adjoint. 
\end{theorem}
\begin{proof}
As the proof follows the ideas of that of Theorem \ref{thm:regularised}, we only sketch the main lines
and leave the details to the reader.

If the operators $(\lambda-\L)^{-\a}\cos(t\sqrt{-\L})$, $t\in\R$, are bounded on 
$L^p(E,\mu)$, then so are the operators
 $$(\lambda-\L)^{-(\a+\frac{1}{2})}\sin(t\sqrt{-\L}) = \int_0^t (\lambda-\L)^{-(\a+\frac{1}{2})}\sqrt{-\L}\cos(s\sqrt{-\L})\,\ds,$$ 
as well as 
$$ (\lambda-\L)^{-(\a+\frac{1}{2})}e^{it\sqrt{-\L}} := (\lambda-\L)^{-(\a+\frac{1}{2})}[\cos(t\sqrt{-\L})  
+ i \sin(t\sqrt{-\L})] .$$
Then also the operators
$$ e^{(s+it)\sqrt{-\L}} =  (\lambda-\L)^{(\a+\frac{1}{2})}e^{s\sqrt{-\L}} \circ (\lambda-\L)^{-(\a+\frac{1}{2})}e^{it\sqrt{-\L}}
$$ are bounded on $L^p(E,\mu)$, in the sense that the right-hand side 
defines a holomorphic extension of the semigroup $(e^{t\sqrt{-\L}})_{t\ge 0}$ 
to the right half-plane $\{\Re z>0\}$. This means that $\sqrt{-\L}$ is sectorial of angle zero.
But $-\L$ is sectorial as well, and, by the general theory of sectorial operators, the angles
of sectoriality are related by $\omega(\sqrt{-\L}) = \frac12 \omega(-\L)$ (see, e.g., 
\cite[Theorem 15.16]{KuWe}). It follows that $-\L$ is sectorial of angle zero.
As we have seen in the proof of Theorem \ref{thm:Hor-iL}, 
this is false unless $p=2$ and $\L$ is self-adjoint.
\end{proof}

Concerning exponential regularisation by $e^{s\L}$, analogous observations as in
Remark \ref{rem:expreg} can be made. We leave this to the interested reader.

\section{Speed of propagation}\label{sec:finite-speed}

It will be useful to make the natural identification
$$ L^2(E,\mu)\oplus L^2(E,\mu;H) = L^2(E,\mu;\C\oplus H).$$
The {\em support} of an element $u = (f,g) \in L^2(E,\mu)\oplus L^2(E,\mu;H)$
will always be understood as the support of the corresponding element in $L^2(E,\mu;\C\oplus H)$.
Thus, $\supp(u) = \supp(f)\cup \supp(g)$.

\begin{definition}\label{def:finite-speed}  Let $\calH$ be any Hilbert space. We say 
that a one-parameter family $(T_t)_{t\in \R}$ 
of bounded operators on $L^2(E,\mu;\calH)$ 
has {\em speed of propagation $\kappa$} if the following holds. 
For all closed subsets $K$ of $E$, all $u\in L^2(E,\mu;\calH)$, and all $t\in\R$,  we have
$$\supp (u) \subseteq K \ \Longrightarrow \ \supp (T_tu) \subseteq K_{\kappa|t|}$$
where $$K_{\kappa|t|}:= \{x\in E: \ \dist(x,K) \le \kappa|t|\}.$$
The family $(T_t)_{t\in \R}$ is said to have {\em infinite speed of propagation} if it does
not propagate at any finite speed.
\end{definition}

In the above, $\dist(x,K) = \inf\{\n x-y\n: \ y\in K\}$. 
Note that $(T_t)_{t\in \R}$ has speed of propagation $\kappa$ if and only if for all subsets $K$ of $E$ and all  
$u, u'\in L^2(E,\mu;\calH)$ 
with supports in $K$ and $\complement  K_{\kappa|t|}$ respectively, 
we have $$\lb T_tu, u'\rb = 0,$$ the brackets denoting the inner product of $L^2(E,\mu;\calH)$.

In the next proposition we consider the case $E = \R^d = H$ and $A = \frac12 I$. the resulting operator 
$\L$ is called the {\em classical Ornstein-Uhlenbeck operator} and is given explicitly as
\begin{align}\label{eq:classOU} 
\L = \frac12\Delta - \frac12 x\cdot\nabla
\end{align}
and the associated invariant measure is the standard Gaussian measure $\gamma$ on $\R^d$,
$$ {\rm d}\gamma(x) = \frac1{(2\pi)^{d/2}}\exp(-\frac12|x|^2)\,{\rm d}x.$$
The semigroup generated  by $\L$ is given by 
 $$ e^{t\L} f(x) = \int_{\R^d} M_t(x,y)f(y)\,{\rm d}y,$$
 where $M$ is the {\em Mehler kernel},
 $$ M_t(x,t) = \frac1{(2\pi)^{d/2}}(1-e^{-2t})^{-d/2}\exp\Big(\displaystyle\!-\!\frac12\frac{|e^{-t} x - y|^2}{1-e^{-2t}}\Big).$$

The following theorem is the Ornstein-Uhlenbeck analogue of the classical fact that 
the Schr\"odinger group $(e^{it\Delta})_{t\in\R}$ on $L^2(\R^d)$ has infinite speed of propagation.

 \begin{theorem}
 Let $\L$ be the classical Ornstein-Uhlenbeck operator on $L^2(\R^d,\gamma)$.
 The $C_0$-group $(e^{it\L})_{t\in\mathbb{R}}$ generated by $i\L$
 has infinite speed of propogation. 
\end{theorem}
\begin{proof}
 It suffices to show that, for some given $t_0>0$, and any $R>0$, there exist compactly supported 
 functions $f,g\in L^2(\R^d,\gamma)$
 whose supports are separated at least by a distance $R$, and which satisfy $\lb e^{it_0 \L}f,g\rb \not=0$.
 
 We take $t_0 := \pi/2$. On the one hand,  by \cite[Proposition 3.9.1]{ABHN} we have 
 $e^{it_0\L}f = \lim_{s\downarrow 0} e^{(s+it_0)\L}f$ in $L^2(\R^d,\gamma)$. 
 On the other hand, for almost all $x\in\R^d$ we have
 \begin{align*} e^{(s+it_0)\L}f(x) & = \int_{\R^d} M_{s+it_0}(x,y)f(y)\,{\rm d}y
 \\ & =\frac1{(2\pi)^{d/2}}{(1+e^{-2s})^{-d/2}} \int_{\R^d} {\exp\Big(\displaystyle\!-\!\frac12\frac{|ie^{-s} x - y|^2}{1+e^{-2s}}\Big)}f(y) \,{\rm d}y
 \end{align*}
 by analytic continuation. For compactly supported $f$ we may use dominated convergence to pass to the limit
 for $s\downarrow 0$ and obtain, for almost all $x\in\R^d$,
 \begin{align*} e^{it_0\L}f(x)
 = (4\pi)^{-d/2}\int_{\R^d} \exp\big(\!-\!\frac14|ix - y|^2\big)f(y) \,{\rm d}y
 \end{align*}
Fix arbitrary $x_0,y_0$ in $\R^d$ satisfying $|x_0-y_0|>R$ 
and let $f_m := \frac{\one_{B(x_0,\frac1m)}}{{|B(x_0,\frac1m)|}}$
and $g_n := \frac{\one_{B(y_0,\frac1n)}}{{|B(y_0,\frac1n)|}}$ for $n,m \in \N$. Then, by continuity, 
\begin{align*}
\lim_{m,n\to\infty} \lb e^{it_0 \L}f_m,g_n\rb 
& = \lim_{m,n\to\infty} (4\pi)^{-d/2}\int_{\R^d} \int_{\R^d} \exp\big(\!-\!\frac14|ix - y|^2\big)f_m(y)g_n(x) \,{\rm d}y\,{\rm d}x
\\ & = (4\pi)^{-d/2}\exp\big(\!-\!\frac14|ix_0 - y_0|^2\big) \not=0.
\end{align*}
It follows, by taking $n,m$ large enough, that $\lb e^{it_0 \L}f_m,g_n\rb \not=0$ , while the supports of $f_{m}$ and $g_{n}$ are 
separated by a distance $\tilde{R}\ge R$.
\end{proof}
\

Using the identity
$$e^{z\sqrt{-\L}}f  =  \frac{1}{\sqrt{\pi}} \int_0^\infty
 \frac{e^{-u}}{\sqrt{u}} \exp\big(\frac{z^2}{4u}\L\big)f \,du, \quad \Re z \ge  0,$$
by a similar argument one shows that the $C_0$-group $(e^{it\sqrt{-\L}})_{t\in\mathbb{R}}$ generated by $i\sqrt{-\L}$
 has infinite speed of propogation. 
 
The main result of this section provides conditions under which the $C_0$-group  $(e^{it\D})_{t\in\R}$  generated by 
$i\D$ has finite speed of propagation;  as an immediate corollary, the cosine family
$(\cos(t\sqrt{-\L})_{t\in\R}$ has finite speed of propagation. 
In addition to Assumption \ref{ass:mu-infty} we need an assumption on
the reproducing kernel Hilbert space $H_\mu$ associated with the Gaussian measure 
$\mu$. Recall that this is the Hilbert space completion of $\Ran(Q_\mu)$, 
where $Q_\mu$ is the covariance operator of $\mu$,
with respect to the norm 
\begin{align}\label{eq:Qmu}
\n Q_\mu x\s\n_{H_\mu}^2 := \lb Q_\mu x\s,x\s\rb \quad  \forall x^{*} \in E^{*} .
\end{align}
 This completion embeds continuously
into $E$. Denoting by $i_\mu: H_\mu\to E$ the embedding mapping, we have  $i_\mu\circ i_\mu^*{} = Q_\mu$. 
For more information we refer the reader to \cite{Bo, Nee-Can}.

\begin{assumption}\label{ass:L-spectral-gap}
$H_\mu$ is densely contained in $H$.
\end{assumption}

By an easy closed graph argument the inclusion mapping $i_{\mu,H}:H_\mu\to H$ is bounded. 

The relevance of this Assumption lies in the fact 
that $H_\mu$ is contained in $H$ if and only if $\L$ has a spectral gap, which in turn is equivalent 
to the validity of the following Poincar\'e inequality: for some (equivalently, for all) $1<p<\infty$
there is a constant $C_p$ such that for all $f\in\Dom_p(\nabla_H)$,
\begin{align}\label{eq:Poincare}\n f - \bar f\n_p \le C_p \n \nabla_H f\n_p,
\end{align}
 with $\bar f:= \int_E f\,\dmu$ (see \cite{CG02, GolNee, NeePoinc}).

If $H_{\mu} \subseteq H$, then the inclusion is dense if and only if $\nabla_H$ is closable
as a densely defined operator from $L^p(E,\mu)$ to $L^p(E,\mu;H)$ for some/all $1\le p<\infty$ \cite[Corollary 4.2]{GGN}.
These equivalent conditions are satisfied if
the semigroup $S$ restricts to a $C_0$-semigroup on $H$ \cite[Theorem 3.5]{GolNee}; the latter is the case
if Assumption \ref{ass:L-analytic} is satisfied \cite[Theorem 3.3]{MaaNee11}.

Now we are ready to state the main result of this section.

\begin{theorem}\label{thm:MainThm} Let Assumptions \ref{ass:mu-infty} and \ref{ass:L-spectral-gap}
hold and let $\L$ be self-adjoint. 
Then the group $(e^{it\D})_{t\in\mathbb{R}}$ on $L^2(E,\mu)\oplus L^2(E,\mu;H)$ 
propagates, at most, with speed $\n i_H\n_{\calL(H,E)}$.
\end{theorem}
 
Our proof of this theorem follows an argument of Morris and M$^{\rm c}$Intosh \cite{McIMor}, which 
in turn is a group analogue of a similar resolvent argument in \cite{AKM}. 
The main difficulty in carrying over the proof to the present situation is to prove 
that suitable Lipschitz functions belong to $\Dom(\nabla_H)$.
 
We begin with some lemmas. It will be understood that 
the assumptions of Theorem \ref{thm:MainThm} are satisfied, although not all assumptions are needed in each lemma.

\begin{lemma}\label{lem:comm} For all real-valued 
$\eta \in\Dom(\nabla_H)$ satisfying $\nabla_H \eta \in L^\infty(E,\mu;H)$
and all $u = (f,g)\in\Dom(\D)$
we have $\eta u = (\etaA f, \etaB g)\in\Dom(\D)$ and the commutator $$[\eta,\D]: u \mapsto \eta \D u - \D(\eta u)$$
extends to a bounded operator on $L^2(E,\mu)\oplus L^2(E,\mu;H)$ 
with norm $$ \n  [\eta ,\D]\n \le \n \nabla_H \eta\n_\infty.$$
This operator is local, with support contained in the support of $\eta$, in the sense that 
$[\eta,\D]u = 0$ whenever ${\rm supp}(u)\cap {\rm supp}(\eta) = \emptyset$. 
Furthermore, $$[\eta,[\eta, \D]]=0.$$
\end{lemma}
\begin{proof}
For all $f\in\Dom(\nabla_H)$ we have (by approximating $\eta$ and $f$ with cylindrical functions)
$\etaA f \in\Dom(\nabla_H)$ and $\nabla_H(\etaA f) = \eta \nabla_H f + (\nabla_H \etaA)f$. Also, for all $f\in\Dom(\nabla_H)$ and $g\in\Dom(\nabla_H\s)$, we have 
$$ \lb \nabla_H f, \etaB g\rb = \lb \etaB\nabla_H f, g\rb = \lb \nabla_H (\etaB f) -  f \nabla_H\etaB,  g\rb,$$
where the brackets denote the inner product 
of $L^2(E,\mu;H)$.
It follows that $\etaB g\in\Dom(\nabla_H\s)$ and $\nabla_H\s (\etaB g) 
= \etaB\nabla_H\s g -\lb \nabla_H \etaB, g\rb_H$; here the brackets $\lb \cdot,\cdot\rb_H$ denote the (pointwise)
inner product of $H$.
Hence, for $u = \bma f \\ g\ema  \in \Dom(\D)$ (that is, $f\in\Dom(\nabla_H)$ and $g\in\Dom(\nabla_H\s)$), 
\begin{align*}
 [\eta,\D]u 
= \left[\begin{matrix}  \etaB \nabla_H\s  g \\ \etaA \nabla_H f \end{matrix}\right] 
-\left[\begin{matrix} \nabla_H\s (\etaB g) \\ \nabla_H (\etaA f) \end{matrix}\right]  
= \left[\begin{matrix} \lb \nabla_H \etaB,g\rb_H \\ -(\nabla_H \etaA)f  \end{matrix}\right]. 
\end{align*}
We infer that $[\eta,\D]$ is bounded and $\n [\eta,\D]\n \le \n \nabla_H \eta\n_{\infty}$. The locality 
assertion is an immediate consequence of the above representation of $[\eta,\D]$.

To prove that $ [\eta ,[\eta , \D]]=0$, just note that
$$[\eta,[\eta , \D]] u = \left[\begin{matrix} \etaA\lb \nabla_H \etaB,g\rb_H 
\\ -\etaB(\nabla_H \etaA)f  \end{matrix}\right]
-\left[\begin{matrix} \lb \nabla_H \etaB,\etaA g\rb_H \\ -(\nabla_H \etaA) \etaB f  \end{matrix}\right] = 0.
$$
\end{proof}

As in \cite{McIMor} we deduce:

\begin{lemma}\label{lem:CommForm}
Under the hypotheses of Theorem~\ref{thm:MainThm}, the following commutator identity holds
for all $t\in\R$, $\eta\in\Dom(\nabla_H)$, and $u\in L^2(E,\mu)\oplus L^2(E,\mu;H)$:
\[
[\eta ,e^{it\D}]u = it \int_0^1 e^{ist\D} [\eta ,\D] e^{i(1-s)t\D}u\,\ds.\]
\end{lemma}

At the heart of the approach in \cite{McIMor} is the following lemma. We include its proof for the sake of completeness.

\begin{lemma}[McIntosh  \&  Morris] Let $u, v\in L^2(E,\mu)\oplus L^2(E,\mu;H) = L^2(E,\mu;\C\oplus H)$ have disjoint supports 
and let $\eta\in\Dom(\nabla_H)$ be a real-valued function satisfying
\begin{align}\label{eq:grad}\hbox{$\eta u = u\,$ and $\,{\eta} v = 0$. }
\end{align}
Then for all $t\in\R$ we have \[|\lb e^{it\D}u,v \rb | \leq   |t|^n\|\nabla_H\eta\|_\infty^n \|u\|_2\|v\|_{2}.
\]
In particular, for  $|t|< 1/\|\nabla_H\eta\|_\infty$ it follows that 
$\lb e^{it\D}u,v \rb = 0.$ 
\end{lemma}
\begin{proof}
To simplify the notation, let $\delta$ be
the {\it derivation} defined by $\delta(S)=[\eta,S]$ and inductively write 
$\delta^k(S) := \delta(\delta^{k-1}(S))$ for the higher commutators,
adopting the convention that $\delta^0(S):=S$. 
Then, for all integers $k\ge 1$,
$$ \lb \delta^k(e^{it\D}) u,v \rb  = \lb \eta \delta^{k-1}(e^{it\D})u - \delta^{k-1}(e^{it\D})\eta u,v \rb 
= -\lb \delta^{k-1}(e^{it\D})u,v \rb,
$$
using the assumptions that $\eta u=u$ and ${\eta} v=0$.
Hence, by induction, 
\begin{equation}\label{easily}
\lb \delta^n(e^{it\D}) u,v \rb =(-1)^n \lb e^{it\D} u,v\rb, \quad n\ge 1.
\end{equation}

On the other hand, using the identity $\delta(ST)=\delta(S)T+S\delta(T)$, Lemma~\ref{lem:CommForm}, 
 and the fact, given by the second assertion in Lemma \ref{lem:comm}, 
that $\delta([\eta, \D])=[\eta, [\eta, \D]]=0$, we obtain
\begin{equation}\label{eq: ind.In}
\delta^{m+1}(e^{it\D})u = it\int_0^1 \sum_{k=0}^{m} {\binom{m}{k}}
\delta^{m-k}(e^{ist\D}) [\eta, \D] \delta^k(e^{i(1-s)t\D})u\,\ds, \quad m\ge 0,
\end{equation}
where
$\displaystyle\binom{m}{k}:=\frac{m!}{k!(m-k)!}$. We now prove by induction that
\begin{equation}\label{eq: ind.NormIn}
\|\delta^n(e^{it\D})\| \leq  |t|^n\,\|[\eta,\D]\|^n, \quad m\ge 0. 
\end{equation}
For $n=0$, this follows from the fact that the 
operators $e^{it\D}$ are unitary. Now
let $m\ge 0$ and suppose that \eqref{eq: ind.NormIn} holds for all integers
$0\le n\leq m$. We then use~\eqref{eq: ind.In} to obtain
\begin{align*}
\|\delta^{m+1}(e^{it\D})\| &\leq |t| \int_0^1 \sum_{k=0}^{m} 
\binom{m}{k} \|\delta^{m-k}(e^{ist\D})\|\, \|[\eta, \D]\|\,
\|\delta^{k}(e^{i(1-s)t\D})\| \, \ds \\
&\leq  |t|^{m+1}\,\|[\eta,\D]\|^{m+1}  \int_0^1 \sum_{k=0}^{m}
{\binom{m}{k}} {s}^{m-k} (1-{s})^k  \, \ds \\
&=  |t|^{m+1}\,\|[\eta,\D]\|^{m+1}  \int_0^1 ({s} + (1-{s}))^{m} \, \ds \\
&=  |t|^{m+1}\,\|[\eta,\D]\|^{m+1} \ .
\end{align*}
 This proves~\eqref{eq: ind.NormIn}.

The lemma now follows by using the estimate \eqref{eq: ind.NormIn} in \eqref{easily} together
with Lemma \ref{lem:comm}.  
\end{proof}

\begin{proof}[Proof of Theorem \ref{thm:MainThm}]
What remains to be proven is that, given $\varepsilon>0$, disjoint closed sets $A$ and $B$ in $E$
can be `separated' by an $\eta \in \Dom(\nabla_H)$, 
in the sense that $\eta \equiv 1$ on $A$ and $\eta\equiv 0$ on $B$,
that satisfies $\nabla_H \eta\in L^\infty(E,\mu;H)$ and
$$\n \nabla_H \eta\n_\infty \le (1+\varepsilon) \n i_H\n/\dist(A,B).$$

It is clear that we can do the separation with bounded Lipschitz functions $f$ whose Lipschitz constant $L$
is at most $(1+\varepsilon)/\dist(A,B)$. To complete the proof, 
we need to show that such functions do indeed belong to $\Dom(\nabla_H)$ and satisfy  
$\n \nabla_H f\n_\infty \le \n i_H\n L.$ This last step is the most important technical difficulty 
that needs to be overcome in order to apply McIntosh and Morris' approach to finite speed of propagation 
in the Ornstein-Uhlenbeck context. We prove it in Theorem \ref{thm:LipH} from the Appendix, as it is of 
independent interest.

\end{proof}

\section{Off-diagonal bounds}\label{sec:DG}

The results of the previous sections will now be applied to obtain $L^{2}-L^{2}$ off-diagonal bounds
for Ornstein-Uhlenbeck operators. Such off-diagonal bounds can be seen as integrated versions of heat kernel bounds, and play a key role in the modern approach to spectral multiplier problems. As can be seen, e.g., in \cite{AKM}, such bounds are particularly useful when dealing with semigroups that do not have standard Calder\'on-Zygmund kernels, but still exhibit a diffusive behaviour. For more information on the role of Davies-Gaffney bounds and finite speed of propagation from the point of view of geometric heat kernel estimates, see e.g. \cite{BCS}. For their use in spectral multiplier theory, see e.g. \cite{COSY}.

We begin with some general observations.
If $-iG$ generates a bounded $C_0$-group $U$ on a Banach space $X$, for any $\phi\in L^1(\R)$ we may define a 
bounded operator $\wh \phi(G)$ by means of the {\em Weyl functional calculus} (see, e.g., \cite{Jeff}):
$$ \wh\phi(G)x := \int_{-\infty}^\infty \phi(t) U(t)x\,\dt, \quad x\in X.$$
When $X$ is a Hilbert space and $G $ is 
self-adjoint, $U$ is unitary and the definition of $\wh\phi(G)$ 
agrees with the one obtained by the spectral theorem:
\begin{align*}
 \int_{-\infty}^\infty \phi(t) U(t)x\,\dt
  & =\int_{-\infty}^\infty \phi(t) \int_{\sigma(G)} e^{-it\lambda} \,{\rm d}E(\lambda)x\,\dt
  \\ & = \int_{\sigma(G)}  \int_{-\infty}^\infty \phi(t) e^{-it\lambda} \,{\rm d}E(\lambda)x
       = \int_{\sigma(G)}  \int_{-\infty}^\infty \wh\phi(\lambda) \,{\rm d}E(\lambda)x.
\end{align*}

As an application of finite speed of propagation we prove, under the assumptions of 
Theorem \ref{thm:MainThm}, some off-diagonal bounds
for the operators $\wh\phi(\D)$ in the self-adjoint case, i.e., where $\D = \bma 0 & \nabla_H\s \\ \nabla_H & 0\ema$.
We have learnt this argument from Alan M$^{\rm c}$Intosh.
The main observation is the following. If $u,v\in L^2(E,\mu)\oplus L^2(E,\mu;H)$ have supports 
separated by a distance $R\n i_H\n$,
we apply the Weyl calculus to $\D$ and note that $\lb e^{-it\D}u,v\rb=0$ for $|t|\le R$ since 
$e^{-it\D}$ propagates with speed at most $\n i_H\n$. As a consequence 
we obtain 
 \begin{align*}
|\lb  \wh\phi(\D) u,v\rb|  
 & = \Big|\int_{|t|\ge R} \phi(t) \lb e^{-it\D}u,v\rb \,\dt\Big|
  \le \Big(\int_{|t|\ge R} |\phi(t)| \,\dt\Big)\n u\n_2\n v\n_2.
\end{align*}
We will work out two special cases where this leads to an interesting explicit estimate.

\begin{example}
\label{ex:alan}
Let $R>0$.
\begin{enumerate}
\item 
If $f\in L^2(E,\mu)$ and $g\in L^2(E,\mu)$ have supports separated by a distance at least $R\n i_H\n$,
then
$$\displaystyle |\lb e^{t\L}f,g\rb| \le  \frac{2t}{\pi R^2}\exp\big(\frac{-R^2}{2t}\big) \n f\n_2 \n g\n_2.$$
\item
If $f\in L^2(E,\mu)$ and $g\in L^2(E,\mu;H)$ have supports separated by a distance at least $R\n i_H\n$,
then
$$\displaystyle |\lb \nabla_H e^{t\L}f,g\rb| 
\le \sqrt{\frac2{\pi t}}  \exp\big(\frac{-R^2}{2t}\big) \n f\n_2 \n g\n_2.$$
The same estimate holds for $e^{t\L}\nabla_H\s$.
\end{enumerate}
\end{example}

\begin{proof}
By the Weyl calculus,
$$ \bma e^{t\L} & 0 \\ 0 & e^{t\underline{\L}}\ema = e^{-\frac12t\D^2} 
= \frac1{\sqrt{2\pi t}} \int_{-\infty}^\infty  e^{-s^2/2t}e^{-is\D} \,\ds$$
and
\begin{align}\label{eq:grad-FT} \bma 0 & e^{t\L}\nabla_H\s \\ \nabla e^{t\L} & 0\ema = \D e^{-\frac12t\D^2} 
= \frac{i}{\sqrt{2\pi t^3}} 
\int_{-\infty}^\infty s e^{-s^2/2t}e^{-is\D} \,\ds,
\end{align}
where we used that $\nabla_H\s e^{t\underline{\L}}=\nabla_H\s e^{t\underline \L\s}  =  
e^{t{\L}}\nabla_H\s$.

The first assertion of the theorem now follows from the theorem via 
 \begin{align*}
|\lb e^{t\L}f,g\rb| 
 =   \big|\big\lb e^{-\frac12t\D^2}\bma f \\ 0\ema, \bma g \\ 0 \ema\big\rb\big| 
& \le  \frac2{\sqrt{2\pi t}} \int_{R}^\infty e^{-s^2/2t} \n f\n_2 \n g\n_2\,\ds
\\ & = \frac2{\sqrt{2\pi}} \int_{R/\sqrt t}^\infty e^{-s^2/2} \n f\n_2 \n g\n_2\,\ds
\\ & \le\sqrt{\frac{2t}{\pi R^2}}e^{-R^2/2t} \n f\n_2 \n g\n_2,
 \end{align*}
where the last inequality follows from a standard estimate for the Gaussian distribution.
Similarly,
\begin{align*}
|\lb \nabla_H e^{t\L}f,g\rb|
 =   \big|\big\lb e^{-\frac12t\D^2}\bma f \\ 0\ema, \bma 0 \\ g \ema\big\rb\big| 
 & \le \frac{2}{\sqrt{2\pi t^3}} \int_{R}^\infty se^{-s^2/2t} \n f\n_2 \n g\n_2\,\ds
\\ & = \sqrt{\frac{2}{\pi t}}  e^{-R^2/2t} \n f\n_2 \n g\n_2.
 \end{align*}
The proof for $e^{t\L}\nabla_H\s$ is similar. 
\end{proof}

Similar results can be obtained by considering other functions $\phi$. For instance, off-diagonal bounds for
$\L e^{t\L}$ may be obtained by taking $\phi(s) =s^2 e^{-s^2/2t}$ and using the identity 
\begin{align*} \bma t\L e^{t\L} & 0 \\ 0 & t\underline \L  e^{t\underline \L} \ema  
& = -\frac{1}{2}t\D^2e^{-\frac12t\D^2}. \end{align*}
We leave the details to the reader.

\subsection{Off-diagonal bounds for resolvents in the non-self-adjoint case}
In the non-self-adjoint case, we cannot make use of finite speed of propagation for 
an underlying group to prove off-diagonal bounds for $(e^{t\L})_{t \geq 0}$.
However, it is possible to use the direct approach from \cite{AKM} (and its refinement 
in \cite{AAM10-2}) to obtain off-diagonal bounds for $((I+t^{2}\L)^{-1})_{t \in \R}$.
We leave the investigation of possible other approaches for $(e^{t\L})_{t \geq 0}$ in the non-self-adjoint case for future work.

We adopt Assumptions \ref{ass:mu-infty}
and \ref{ass:L-spectral-gap}. We do not assume $\L$ to be self-adjoint, so $i\D$ may fail to
generate a $C_0$-group on $L^2(E,\mu)\oplus L^2(E,\mu;H)= L^2(E,\mu;\mathbb{C}\oplus H)$. 
Nevertheless, $\D$ does enjoy some good properties; for instance 
it is bisectorial on  $L^2(E,\mu;\mathbb{C}\oplus H)$  and therefore
the quantity
\begin{align}\label{eq:Mbisect} 
M := \sup_{t\in\R} \n (I-it\D)^{-1}\n
\end{align}
 is finite.
This follows from the general operator-theoretic framework presented in \cite{AKM}.

\begin{proposition}\label{prop:OD-resolvent}
 Let Assumptions \ref{ass:mu-infty} and \ref{ass:L-analytic} hold.
Suppose $u,v\in L^2(E,\mu;\C\oplus H)$ have disjoint supports at a distance greater than $R$. Then 
$$
|\lb (I+it\D)^{-1}u,v\rb| \leq C \exp(-\alpha R/|t|)\n u\n_2\n v\n_2,
$$
for some $\alpha,C>0$ independent of $u,v$ and $R,t$.
\end{proposition}
\begin{proof}
The proof is a straight forward adaptation of \cite[Proposition 5.1]{AAM10-2}, and is included for the sake of completeness.

By the uniform boundedness of the operators $R_t:=(I-it\D)^{-1}$, $t\in\R$, it suffices to prove the estimate
in the statement of the proposition
for $|t|< \alpha R$, where $\alpha>0$ is a positive constant to be chosen in a moment.

Let $u\in L^2(E,\mu;\C\oplus H)$ be supported in a set $B\subseteq E$ and let $A\subseteq E$ be another set
such that $\dist(A,B) \ge R$.
Define
$$\wt A := \big\{x \in E: \ \dist(x,A) < \frac12 \dist(x,B)\big\}.$$
Note that $\dist(\wt A,B) \ge \frac12 \dist(A,B)$.

Let $\varphi: E \to [0, 1]$ be a bounded Lipschitz function with support in $\wt A$
such that $\varphi|_{A} \equiv 1$,  $\varphi|_B \equiv 0$, and whose Lipschitz constant is at most $4/R$.
By Theorem \ref{thm:LipH}, $\varphi\in\Dom(\nabla_H)$ and $\n \nabla_H \varphi\n \le 4\n i_H\n/R$.

Set $\eta:= \exp(\alpha R \varphi/|t|) -1$. Then,  for all $x \in A$,
$$\eta(x) = \exp(\alpha R/|t|) -1 \ge \frac12\exp(\alpha R/|t|)$$ 
(recall the assumption $|t|< \alpha R$) and $\eta|_B \equiv 0$. Hence,
\begin{align}\label{eq:eta}
\frac12\exp(\alpha R/|t|)\n R_t u \n_{L^2(A,\mu;\C\oplus H)} \le \n\eta R_t u\n_2 = \n [\eta,R_t]u\n_2
\end{align}
using that $\eta u = 0$ by the support properties of $\eta$ and $u$.

It is elementary to verify the commutator identity
$$[\eta,R_t] = it R_t[\eta,\D] R_t.$$
Moreover, using Leibniz rule  (see the proof of Lemma \ref{lem:comm}),  we have
\begin{align*}
 [\eta,\D]v = [\exp(\alpha R \varphi/|t|) -1,\D]v
 &  = (\D\exp(\alpha R \varphi/|t|))v
  =  m\exp(\alpha R \varphi/|t|)v,
\end{align*}
where $m$ is supported on $\wt A$ and satisfies (cf. \ref{lem:comm})
$\n m\n_\infty \le C\a R \n  \nabla_H  \varphi\n/|t| \le 4C\alpha \n i_H\n  /|t|$.
Therefore,
\begin{align*}
    \n [\eta,R_t] u\n_{L^2(\wt A,\mu|_{\wt A};\C\oplus H)}
& = |t| \n R_t [\eta,\D] R_t u\n_{2}
\\ & \le  4MC\alpha  \n i_H\n  \n \exp(\alpha R \varphi/|t|) R_t u\n_{2}
\\ & \le  4MC\alpha  \n i_H\n  \Big(\n\eta R_t u\n_{L^2(\wt A,\mu|_{\wt A};\C\oplus H)}
                                 + \n R_t u\n_2\Big)
\end{align*}
where $M$ is defined by \eqref{eq:Mbisect}.
The choice $\alpha = (8MC \n i_H\n  )^{-1}$, in combination with \eqref{eq:eta}, gives
$$\frac12\exp(\alpha R/|t|)\n R_t u \n_{L^2(A,\mu;\C\oplus H)}
\le  \n [\eta,R_t] u\n_{2} \leq \n R_t u\n_2 \le M\n u\n_2.$$
\end{proof}

\begin{remark}
With the same proof, Proposition \ref{prop:OD-resolvent}  holds in the more general 
context of elliptic divergence-form operators on abstract Wiener spaces considered in 
\cite{MN}.
\end{remark}

Since $$ (I+ it \D)^{-1} +  (I- it \D)^{-1} =  2(I+ t^2 \D^2)^{-1} 
=  2\bma (I- 2t^2 \L)^{-1} & 0 \\ 0 &  (I -2t^2 \L)^{-1}\ema,$$ we have the following corollary.

\begin{corollary}
Let Assumptions \ref{ass:mu-infty} and \ref{ass:L-analytic} hold.
Suppose $u,v\in L^2(E,\mu;\C\oplus H)$ have disjoint supports at a distance greater than $R$. Then 
$$
|\lb (I-t^{2}\L)^{-1}u,v\rb| \leq C \exp(-\alpha R/|t|)\n u\n_2\n v\n_2,
$$
for some $\alpha,C>0$ independent of $u,v$ and $R,t$.
\end{corollary}

\section{Appendix: $H$-Lipschitz functions}

It is assumed that Assumptions \ref{ass:mu-infty} and
and \ref{ass:L-spectral-gap} hold. Our aim is to prove that under these conditions, 
bounded Lipschitz functions on $E$ (and more generally, bounded
$H$-Lipschitz functions on $E$) belong to $\Dom(\nabla_{H})$ with a suitable bound; this result was  
needed in the proof of Theorem \ref{thm:MainThm}.
We point our that this result becomes trivial in the case $E = \R^d = H$,
which is the setting for studying the classical Ornstein-Uhlenbeck operator on $L^p(\R^d,\gamma)$
(see \eqref{eq:classOU}). Readers whose main interests concern this particular case
will therefore not need the result presented here.

We recall some further standard facts about reproducing kernel Hilbert spaces. The reader is referred to 
\cite{Bo, Nee-Can} for the proofs and more details.
Recall that $H_\mu$ denotes the reproducing kernel Hilbert space associated with the invariant measure 
$\mu$ (see \eqref{eq:Qmu}) and that $i_{\mu}: H_\mu\to E$ denotes the inclusion mapping.
Since $\mu$ is Radon, the Hilbert space $H_\mu$ is separable. When no confusion can arise we will suppress
the mapping $i_\mu$ from our notations and identify $H_\mu$ with its image in $E$.

The mapping $$\phi: i_\mu\s x\s\mapsto \lb \cdot, x\s\rb$$ extends to an isometric embedding
of $H_\mu$ into $L^2(E,\mu)$. In what follows we shall write $\phi_h := \phi h$ for the image in 
$L^2(E,\mu)$ of a vector $h\in H_\mu$.
By the {\em Karhunen-Lo\`eve decomposition} (see \cite[Corollary 3.5.11]{Bo}), if 
$(h_n)_{n\ge 1}$ is an orthonormal basis for $H_\mu$, then for $\mu$-almost all $x\in E$ we have
$$ x = \sum_{n=1}^\infty \phi_{h_n}(x) h_n$$
with convergence both $\mu$-almost surely in $E$ and in the norm of $L^2(E,\mu)$.
 We may furthermore choose the vectors $h_n\in H_\mu$ 
in such a way that $h_n = i_\mu\s x_n\s$ for 
suitable $x_n\s\in E\s$. In doing so, this exhibits the function $x\mapsto x$ as the limit (for $N\to\infty$)
of the cylindrical functions 
$\sum_{n=1}^N \lb x, x_n\s\rb h_n$. 

The next lemma relates functions which have pointwise directional derivatives in the direction of $H$ 
with functions in the domain of the directional gradient $\nabla_H$. To this end we recall that 
a function $f:E\to \R$ is said to be {\em G\^ateaux differentiable in the direction of $H$ at a point $x\in E$}
if there exists an element $h(x)\in H$, the {\em G\^ateaux derivative of $f$ at the point $x$} such that for all $h\in H$ we have 
$$ \lim_{t\downarrow 0} \frac1t (f(x+th)-f(x)) = \lb h,h(x)\rb.$$
The function $f$ is said to be {\em G\^ateaux differentiable in the direction of $H$}
if it is G\^ateaux differentiable in the direction of $H$ at every point $x\in E$. The resulting
function which assigns to each point $x\in E$ the G\^ateaux derivative of $f$ at the point $x$
is denoted by $D_H f: E\to H$.

\begin{lemma} If $f:E\to \R$ is uniformly bounded and G\^ateaux differentiable  in the direction of $H$, with 
bounded and strongly measurable 
derivative $D_H f$,
then $f\in \Dom(\nabla_H)$ and $\nabla_H f = D_H f$.
\end{lemma}
\begin{proof}
Let $(h_n)_{n\ge 1}$ be a fixed orthonormal basis in $H_\mu$, chosen
in such way that $h_n = i_\mu\s x_n\s$ for suitable $x_n\s\in E\s$.
For all $N\ge 1$ we have, for $\mu$-almost all $x\in E$, 
$$f_N(x):=f\Big(\sum_{n=1}^N \phi_{h_n}(x)h_n\Big)=
\psi_N(\phi_{h_1}(x),\dots,\phi_{h_N}(x)),$$
where the function 
\begin{align}\label{eq:psi-f} \psi_N(t_1,\dots,t_N) =  f\Big(\sum_{n=1}^N t_n h_n\Big)\end{align}
belongs to $ C^1_{\rm b}(\R^N)$. Since we are assuming that $h_n = i_\mu\s x_n\s$,
$f_N$ belongs to $\Dom(\nabla_H)$ and for $\mu$-almost all $x\in E$ we have  
\begin{align*}\nabla_H f_{N}(x)
& = \sum_{j=1}^N \frac{\partial \psi_N}{\partial t_j}(\phi_{h_1}(x),\dots,\phi_{h_N}(x)) h_j
=  \sum_{j=1}^N\Big\lb h_j, D_H f\Big(\sum_{n=1}^N \phi_{h_n}(x) h_n\Big)\Big\rb h_j
\end{align*}
noting that
\begin{align*} \frac{\partial \psi_N}{\partial t_j}(t_1,\dots,t_N) 
& = \lim_{\tau\to 0} 
\frac1\tau \Big[f\Big(\sum_{n=1}^N (t_n+\delta_{jn}\tau) h_n\Big) - f\Big(\sum_{n=1}^N t_n h_n\Big)\Big]
\\ & =  \Big\lb h_j, D_H f\Big(\sum_{n=1}^N t_n h_n\Big)\Big\rb, 
\end{align*}
with $\delta_{jn}$ the Kronecker symbol. 

To finish the proof we will show three things: 
\begin{enumerate}[\rm(i)] 
\item $\lim_{N\to\infty} f_N = f$ in $L^2(E,\mu)$;
\item the sequence $(\nabla_H f_N)_{n\ge 1}$ is Cauchy in $L^2(E,\mu;H)$; 
\item $\mu$-almost everywhere we have $\lim_{N\to \infty} \nabla_H f_N = D_H f$.
\end{enumerate}
Once we have this,
the closedness of $\nabla_H$ will imply that $f\in\Dom(\nabla_H)$ and $\nabla_H f = \lim_{N\to\infty}
\nabla_H f_N = D_H f$.

\smallskip
(i): \ The first claim follows by dominated convergence.

\smallskip
(ii) and (iii): \ Fix integers $M\ge N\ge 1$. Then,
\begin{align*}
&  \Big\n \sum_{j=1}^M\Big\lb h_j, D_H f\Big(\sum_{n=1}^M \phi_{h_n}(\cdot) h_n\Big)\Big\rb h_j
  -  \sum_{j=1}^N\Big\lb h_j, D_H f\Big(\sum_{n=1}^N \phi_{h_n}(\cdot) h_n\Big)\Big\rb h_j\Big\n_{L^2(E,\mu;H)}
\\ & \qquad
\le \Big\n \sum_{j=1}^M\Big\lb h_j,\Big[D_H f\Big(\sum_{n=1}^M \phi_{h_n}(\cdot)h_{n}\Big) -D_H f(\cdot)\Big]\Big\rb h_j\Big\n_{L^2(E,\mu;H)}
\\ & \qquad\qquad + 
\Big\n \sum_{j=1}^M\big\lb h_j,D_H f(\cdot)\big\rb h_j
  -  \sum_{j=1}^N\big\lb h_j,D_H f(\cdot)\big\rb h_j\Big\n_{L^2(E,\mu;H)}
\\ & \qquad\qquad +
\Big\n \sum_{j=1}^N\Big\lb h_j,\Big[D_H f\Big(\sum_{n=1}^N \phi_{h_n}(\cdot) h_n\Big) - D_H f(\cdot)\Big]\Big\rb h_j\Big\n_{L^2(E,\mu;H)}
\\ & \qquad =: {\rm (I)} + {\rm (II)} + {\rm (III)}.  
\end{align*} 
To deal with (II) we note that from $H_\mu\embed H$ it follows that 
$f$ has a bounded G\^ateaux derivative in the direction of $H_\mu$, given by 
$D_{H_\mu}f = i_{\mu,H}\s D_H f$, where $i_{\mu,H}$ is the embedding mapping of $H_\mu$ into $H$. 
Now, since $(h_n)_{n\ge 1}$ is an orthonormal basis in $H_\mu$, for $\mu$-almost all $x\in E$ we have
\begin{align*}
  \lim_{K\to\infty} \sum_{j=1}^K\lb D_{H} f(x),h_j\rb_H h_j 
   & = \lim_{K\to\infty} \sum_{j=1}^K\lb i_{\mu,H}\s D_{H} f(x),h_j\rb_{H} h_j 
\\ & = \lim_{K\to\infty} \sum_{j=1}^K\lb D_{H_\mu} f(x),h_j\rb_{H_\mu} h_j =   D_{H_\mu} f(x) 
\end{align*}
with convergence in $H_\mu$, hence in $H$.
Convergence in
$L^2(E,\mu;H)$ then follows by dominated convergence, noting that $ D_{H_\mu} f$ is uniformly bounded 
as an $H_\mu$-valued function, hence also as an $H$-valued function.

Convergence of (I) and (III) follows in the same way, now using that 
\begin{align*}
\ &  \Big\n \sum_{j=1}^K\Big\lb h_j, \Big[D_H f\Big(\sum_{n=1}^{K} \phi_{h_n}(\cdot) h_n\Big) 
      - D_H f(\cdot)\Big]\Big\rb_H h_j\Big\n_{L^2(E,\mu;H)}
\\ & \qquad 
\le \n i_{\mu,H}\n  \Big\n \sum_{j=1}^K\Big\lb h_j,\Big[D_{H_\mu} f\Big(\sum_{n=1}^K \phi_{h_n}(\cdot) h_n\Big) 
      - D_{H_\mu} f(\cdot)\Big]\Big\rb_{H_\mu} h_j\Big\n_{L^2(E,\mu;H_\mu)}
\\ & \qquad
\le \n i_{\mu,H}\n  \Big\n D_{H_\mu} f\Big(\sum_{n=1}^K \phi_{h_n}(\cdot) h_n\Big) 
      - D_{H_\mu} f(\cdot)\Big\n_{L^2(E,\mu;H_\mu)},
\end{align*}
and the right-hand side tends to $0$ as $K\to\infty$ by dominated convergence,
since $\sum_{n=1}^K \phi_{h_n}(x) h_n\to x$ for $\mu$-almost all $x\in E$ and
the function $D_{H_\mu} f = i_{\mu,H}\s D_H f$ is uniformly bounded.  
\end{proof}

We now define
${\rm Lip}_{H}(E)$ as the vector space of all measurable functions that are 
Lipschitz continuous in the direction of $H$, i.e., for which there exists a finite constant $L_f(H)$  
such that $$ \|f(x+h) - f(x)\| \le L_f(H)\n h\n_H\quad \forall x\in E.$$
Note that we take norms in $H$ on the right-hand side. Obviously, every $f\in {\rm Lip}(E)$ belongs to 
${\rm Lip}_{H}(E)$, since 
$$ \n f(x+h) - f(x)\n \le L_f\n h\n_E  \le L_f\n i_H\n_{\calL(H,E)} \n h\n_H.$$
Here, $L_f$ is the Lipschitz constant of $f$ and $i_H$ is the embedding of $H$ into $E$.
It is also easy to see that if $f:E\to \R$ has a uniformly bounded G\^ateaux derivative in the direction of $H$,
then $f\in {\rm Lip}_{H}(E)$ with constant $L_f(H) \le \n D_H f\n_\infty$.

\begin{theorem}\label{thm:LipH} Let Assumptions \ref{ass:mu-infty} and
and \ref{ass:L-spectral-gap} hold.
 If $f\in {\rm Lip}_{H}(E)$ is uniformly bounded and has $H$-Lipschitz constant $L_f(H)$, then 
 $f\in\Dom(\nabla_H)$, $\nabla_{H} f \in L^{\infty}(E,\mu)$, and $\n \nabla_H f\n_{\infty} \le L_f(H)$.  
\end{theorem}
\begin{proof}
It follows from \cite[Theorem 5.11.2]{Bo} and the observation following it
that $f$ is G\^ateaux differentiable in the direction of $H$ $\mu$-almost everywhere,
with derivative satisfying $\n D_H f\n \le L_f(H)$ $\mu$-almost everywhere.
This derivative is weakly measurable, as each $\lb D_H f,x\s\rb$ is
the almost everywhere limit of continuous difference quotients. Since $H$ is separable,
the Pettis Measurability theorem (see \cite[Section 2]{DU77} implies that $D_H f$ is strongly measurable.
 Now the result follows from the previous lemma.
\end{proof}

\medskip

{\em Acknowledgement} -- We thank Alan M$^{\rm c}$Intosh for showing us 
the method of proof in Example \ref{ex:alan}, drawing our attention to the paper \cite{AAM10}, and suggesting that, at least in the finite dimensional case, his approach to finite speed of propagation from \cite{McIMor} should apply to Ornstein-Uhlenbeck operators.
We thank Gilles Lancien for a helpful discussion about Lipschitz bump functions in Banach spaces.
The major part of this paper was written during a visit of the first-named author at the Australian National University.
He wishes to thank his colleagues at the Mathematical Sciences institute for the enjoyable atmosphere
and kind hospitality.

\bibliographystyle{amsplain}
\bibliography{finite-speed}

\end{document}